\documentclass{amsart}
\usepackage{ amsmath, amsthm, amsfonts, amssymb, color}
 \usepackage{mathrsfs}
\usepackage{amsfonts, amsmath}
 \usepackage[colorlinks, citecolor=blue,pagebackref,hypertexnames=false]{hyperref}
 \allowdisplaybreaks
 \usepackage{pgf,tikz}

\begin{document}
\addtolength{\parskip}{8pt}
\parindent0pt

\newcommand\C{{\mathbb C}}

\newtheorem{thm}{Theorem}[section]
\newtheorem{prop}[thm]{Proposition}
\newtheorem{cor}[thm]{Corollary}
\newtheorem{lem}[thm]{Lemma}
\newtheorem{lemma}[thm]{Lemma}
\newtheorem{exams}[thm]{Examples}
\theoremstyle{definition}
\newtheorem{defn}[thm]{Definition}
\newtheorem{rem}[thm]{Remark}
\newcommand\RR{\mathbb{R}}
\newcommand{\la}{\lambda}
\def\RN {\mathbb{R}^n}
\newcommand{\norm}[1]{\left\Vert#1\right\Vert}
\newcommand{\abs}[1]{\left\vert#1\right\vert}
\newcommand{\set}[1]{\left\{#1\right\}}
\newcommand{\Real}{\mathbb{R}}
\newcommand{\supp}{\operatorname{supp}}
\newcommand{\card}{\operatorname{card}}
\renewcommand{\L}{\mathcal{L}}
\renewcommand{\P}{\mathcal{P}}
\newcommand{\T}{\mathcal{T}}
\newcommand{\A}{\mathbb{A}}
\newcommand{\K}{\mathcal{K}}
\renewcommand{\S}{\mathcal{S}}
\newcommand{\blue}[1]{\textcolor{blue}{#1}}
\newcommand{\red}[1]{\textcolor{red}{#1}}
\newcommand{\Id}{\operatorname{I}}
\newcommand\wrt{\,{\rm d}}
\def\SH{\sqrt {H}}

\newcommand{\rn}{\mathbb R^n}
\newcommand{\de}{\delta}
\newcommand{\tf}{\tfrac}
\newcommand{\ep}{\epsilon}
\newcommand{\vp}{\varphi}

\newcommand{\mar}[1]{{\marginpar{\sffamily{\scriptsize
        #1}}}}
\newcommand{\li}[1]{{\mar{LY:#1}}}
\newcommand{\el}[1]{{\mar{EM:#1}}}
\newcommand{\as}[1]{{\mar{AS:#1}}}
\newcommand\CC{\mathbb{C}}
\newcommand\NN{\mathbb{N}}
\newcommand\ZZ{\mathbb{Z}}
\renewcommand\Re{\operatorname{Re}}
\renewcommand\Im{\operatorname{Im}}
\newcommand{\mc}{\mathcal}
\newcommand\D{\mathcal{D}}
\newcommand{\al}{\alpha}
\newcommand{\nf}{\infty}
\newcommand{\comment}[1]{\vskip.3cm
	\fbox{%
		\color{red}
		\parbox{0.93\linewidth}{\footnotesize #1}}
	\vskip.3cm}

\newcommand{\disappear}[1]

\numberwithin{equation}{section}
\newcommand{\chg}[1]{{\color{red}{#1}}}
\newcommand{\note}[1]{{\color{green}{#1}}}
\newcommand{\later}[1]{{\color{blue}{#1}}}
\newcommand{\bchi}{ {\chi}}

\numberwithin{equation}{section}
\newcommand\relphantom[1]{\mathrel{\phantom{#1}}}
\newcommand\ve{\varepsilon}  \newcommand\tve{t_{\varepsilon}}
\newcommand\vf{\varphi}      \newcommand\yvf{y_{\varphi}}
\newcommand\bfE{\mathbf{E}}
\newcommand{\ale}{\text{a.e. }}

 \newcommand{\mm}{\mathbf m}
\newcommand{\Be}{\begin{equation}}
\newcommand{\Ee}{\end{equation}}

\title[Almost everywhere convergence of Bochner-Riesz means]
{ Almost everywhere convergence of Bochner-Riesz means  for the Hermite operators}

 \author{Peng Chen}
 \author{Xuan Thinh Duong}
 \author{Danqing He}
 \author{Sanghyuk Lee}
 \author{Lixin Yan}
 \address{Peng Chen, Department of Mathematics, Sun Yat-sen
 University, Guangzhou, 510275, P.R. China}
 \email{chenpeng3@mail.sysu.edu.cn}
 \address{Xuan Thinh Duong, Department of Mathematics, Macquarie University, NSW 2109, Australia}
\email{xuan.duong@mq.edu.au}
 \address {Danqing He, School of Mathematical Sciences, Fudan University, Shanghai,  200433, P.R. China}
 \email{hedanqing@fudan.edu.cn}
 \address{Sanghyuk Lee,  School of Mathematical Sciences, Seoul national University, Seoul 151-742, Republic  of Korea}
\email{shklee@snu.ac.kr}
 \address{Lixin Yan, Department of Mathematics, Sun Yat-sen   University,
 Guangzhou, 510275, P.R. China}
 \email{mcsylx@mail.sysu.edu.cn}

\date{\today}

\begin{abstract}
Let $H = -\Delta + |x|^2$ be the Hermite operator  in
${\mathbb R}^n$. In this paper we study  almost everywhere convergence of the Bochner-Riesz means
associated with $H$
 which is defined by
 $S_R^{\lambda}(H)f(x)  = \sum\limits_{k=0}^{\infty} \big(1-{2k+n\over R^2}\big)_+^{\lambda} P_k f(x).$
 Here $P_k f$ is the $k$-th Hermite  spectral projection operator.
For $2\le p<\infty$,   we prove that
    $$
    \lim\limits_{R\to \infty} S_R^{\lambda}(H) f=f
    \ \   \  \text{a.e.}
$$
 for   all $f\in L^p(\mathbb R^n)$ provided that $
 \lambda>  \lambda(p)/2$ and  $\lambda(p)=\max\big\{ n\big({1/2}-{1/p}\big)-{1/ 2}, \, 0\big\}.
 $
 Conversely, we also show the convergence  generally fails if $\lambda< \lambda(p)/2$ in the sense that
   there is an $f\in L^p(\mathbb R^n)$ for $2n/(n-1)\le p$ such that
 the convergence fails.  This is in surprising contrast with  a.e. convergence of the classical Bochner-Riesz means for the Laplacian.
 For $n\geq 2$ and $p\ge 2$ our result  tells  that  the critical  summability index for a.e. convergence for  $S_R^{\lambda}(H)$
 is  as small as only  the \emph{half}  of
 the critical index for a.e. convergence of the classical Bochner-Riesz means.
When $n = 1$,  we show a.e. convergence  holds
for $f\in L^p({\mathbb R})$
with $  p\geq 2$ whenever $\lambda>0$.
 Compared with 	the classical result due to Askey and Wainger who showed the optimal $L^p$ convergence for  $S_R^{\lambda}(H)$ on ${\mathbb R}$
  we only need smaller summability index for a.e.   convergence.
	   \end{abstract}
\maketitle


\section{Introduction}
 \setcounter{equation}{0}
 Convergence of Bochner-Riesz means of Fourier transform in the $L^p$ spaces is one of  the most fundamental problems in classical harmonic analysis.
For $\lambda\ge 0$ and $R>0$,  the  classical Bochner-Riesz means  for the Laplacian on $\RN$ are  defined by
\begin{eqnarray}
\label{e1.1}
{S^{\lambda}_{R}f}(x)
=\int_{\mathbb R^n} e^{2\pi ix\cdot \xi} \left(1-{|\xi|^2\over R^2}\right)_+^{\lambda} \widehat{f}(\xi) \,d\xi,  \quad \forall{\xi \in \RN}.
\end{eqnarray}
Here $t_+=\max\{0,t\}$ for $t\in \mathbb R$ and $\widehat{f}\,$ denotes the Fourier  transform  of $f$.
The $L^p$ convergence of $S^\lambda_R f\to f$ as $R\to \infty$ is equivalent to the $L^p$ boundedness of the operator
 $S^\lambda:=S^\lambda_1$, and the longstanding open problem known as \emph{the Bochner-Riesz conjecture} is that,
 for $1\leq p \leq \infty $ and $p\neq 2$, $S^\lambda $ is bounded on $L^p(\mathbb R^n)$ if and only if
\begin{eqnarray} \label{e1.2}
\lambda>\lambda(p):=\max\Big\{ n\Big|{1\over 2}-{1\over p}\Big|-{1\over 2}, \, 0\Big\}.
\end{eqnarray}
It was shown by Herz \cite{He} that  the condition \eqref{e1.2} on  $\lambda$  is
necessary for $L^p$ boundedness of $S^\lambda$.  Carleson and Sj\"olin \cite{CSj} proved the conjecture
 when $n=2$.  Afterward, substantial progress has been
made in higher dimensions, for example see  \cite{tvv, Lee,  BoGu, ghi, Ta3} and references therein.
However, the conjecture still remains open for $n\ge 3$.
 Concerning pointwise convergence,  Carbery, Rubio de Francia and Vega \cite{CRV}
 showed  a.e. convergence with the sharp summability exponent  for  all $f\in L^p(\RN)$,
$$
 \lim_{R\to \infty} S_R^{\lambda}  f=f \ \  {\rm  a.e.}
 $$
provided  $p\geq 2$
and
$
 \lambda>\lambda(p).
$
When $n=2$ the result was previously obtained  by Carbery  \cite{Ca} who proved  the sharp $L^p$ estimates for the maximal Bochner-Riesz means.
Also, see  \cite{C1}  for earlier  partial result based on the maximal Bochner-Riesz estimate in
higher dimensions. Regarding the most recent result for the maximal Bochner-Riesz estimate,
we refer the reader to \cite{Lee2}.

It is remarkable that  the result by  Carbery et al. \cite{CRV}  settled the a.e. convergence problem up to the sharp
 index $\lambda(p)$ for $2\le p\le \infty$. There are also results at the critical
  exponent, i.e., $\lambda=\lambda(p)$ (for example,  see  \cite{A, LS}). It should be mentioned that
   almost everywhere convergence of $S^{\lambda}_Rf$ with $f\in L^p$,  $1<p<2,$
   exhibits  different nature  and few results are known in this direction except when dimension $n=2$ (\cite{LW, Ta1, Ta2}).

\subsection*{Bochner-Riesz means
for
   the Hermite operator}
In this paper we  are concerned with almost everywhere convergence of  Bochner-Riesz means
for
   the Hermite operator $H$  on $ \RR^n$, which is defined by
   \begin{eqnarray}\label{eecc}
H=-\Delta + |x|^2 =-\sum_{i=1}^n {\partial^2\over \partial x_i^2} + |x|^2, \quad x=(x_1, \cdots, x_n).
\end{eqnarray}
The operator $H$ is non-negative and selfadjoint with respect to the Lebesgue measure
on $\RN$. For each non-negative integer $k$, the Hermite polynomials $H_k(t) $ on $\RR$ are
defined by $H_k(t)=(-1)^k e^{t^2} {d^k\over d t^k} \big(e^{-t^2}\big)$, and  the Hermite functions
$h_k(t):=(2^k k !  \sqrt{\pi})^{-1/2} H_k(t) e^{-t^2/2}$, $k=0, 1, 2, \ldots$ form an orthonormal basis
of $L^2(\mathbb R)$.
For
any multiindex $\mu\in {\mathbb N}^n_0$,
 the $n$-dimensional Hermite functions are given by tensor product of the one dimensional Hermite functions:
\begin{eqnarray}\label{ephi}
\Phi_{\mu}(x)=\prod_{i=1}^n h_{\mu_i}(x_i), \quad \mu=(\mu_1, \cdots, \mu_n).
\end{eqnarray}
Then the functions $\Phi_{\mu}$ are eigenfunctions for
the Hermite operator with eigenvalue $(2|\mu|+n)$ and $\{\Phi_{\mu}\}_{\mu\in \mathbb N_0^n}$ form a complete orthonormal
system in $L^2({\RN})$.
Thus, for  every $f\in L^2(\RN)$  we have  the Hermite expansion
\begin{eqnarray} \label{e1.3}
f(x)=\sum_{\mu}\langle f, \Phi_{\mu}\rangle \Phi_\mu(x)=\sum_{k=0}^{\infty}P_kf(x),
\end{eqnarray}
where $P_k$ denotes the Hermite projection operator given by
\begin{eqnarray} \label{e1.5}
P_kf(x)=\sum_{|\mu|=k}\langle f, \Phi_{\mu}\rangle\Phi_\mu(x).
\end{eqnarray}
For $R>0$ the  Bochner-Riesz means for $H$ of order $\lambda\geq 0$   is defined by
\begin{eqnarray}\label{e1.4}
S_R^{\lambda}(H)f(x)
 =
\sum_{k=0}^{\infty} \left(1-{2k+n\over R^2}\right)_+^{\lambda} P_k f(x).
\end{eqnarray}
The assumption $\lambda\geq 0$ is necessary  $S_R^{\lambda}(H)$  to be   defined for all  $R>0$.\footnote{Note that
$S_R^{\lambda}(H)f$ can not be defined with $R^2=2k+n$ if  $\lambda<0$.}

Concerning the  $L^p$ convergence of $S_R^{\lambda}(H)f$, uniform $L^p$ boundedness of  $S_R^{\lambda}(H)$ has been studied by
a number of  authors.
In one dimension, it is known  \cite{AW, Th2} that  if $\lambda>1/6$,  $S_R^{\lambda}(H)$ is uniformly
 bounded on $L^p$ for $1\le p\le \infty$ and,  for $1/6> \lambda\ge 0$ and $1\le p\le \infty$,   $S_R^{\lambda}(H)$ is uniformly bounded
on $L^p(\RR)$  if and only if $\lambda>(2/3) |1/p -1/2|  -1/6$.  In higher dimensions ($n\geq 2$) the $L^p$ boundedness of
$S_R^{\lambda}(H)$ is not so well understood yet.
When $\lambda>(n-1)/2$  Thangavelu \cite{Th22} showed  uniform boundedness of
 $S_R^{\lambda}(H)$  on  $L^p$ for $1\le p\le \infty.$ In particular,  $S_R^{\lambda}(H)$ converges to $f$ in $L^1(\RN)$
 if and only if $\lambda>(n-1)/2$.  For $0\le \lambda\le (n-1)/2$ and $1\le p\le \infty$, $p\neq 2$,  it still seems natural to conjecture that
$S_R^{\lambda}(H)$ is uniformly bounded on $L^p(\RN)$   if and only if $\lambda>\lambda(p)$ (see~\cite[p.259]{Th4}).
 Thangavelu also showed $\|S_R^{\lambda}(H)f\|_p\le C\|f\|_p$ if and only if  $\lambda>\lambda(p)$
 under the assumption that $f$ is radial, thus  the   condition
$\lambda>\lambda(p)$  is necessary for $L^p$ boundedness of $S^\lambda_R(H)$.
The necessity of the condition $\lambda>\lambda(p)$ for $L^p$ boundedness can also be shown by the transplantation result in  \cite{KST} which deduces
the $L^p$ boundedness of  $S_R^{\lambda}$ from that of $S_R^{\lambda}(H)$.   Karadzhov  \cite{Kar}
 verified the conjecture in the range  $1\leq p\leq 2n/(n+2).$ The boundedness for $ p\in [2n/(n-2), \infty]$
 follows from duality.  However, it   remains open to see if the conjecture is true in the range $2n/(n+2)<p\leq 2n/(n+1)$.

\subsection*{Almost everywhere convergence} Concerning a.e.  convergence of $S_R^{\lambda}(H)f$,
it is known \cite{Th2, Th22} (see also \cite[Chapter 3]{Th3}) that  $S^{\!\lambda}_{\!R}(H) f$ converges  to $f$ a.e.
 for every $f\in L^p(\RN), 1\leq p< \infty,$ whenever $\lambda >(3n-2)/6$.
Recently, Chen, Lee, Sikora and Yan  \cite{CLSY}  studied   $L^p$ boundedness of the maximal Bochner-Riesz means
for the Hermite operator $H$ on ${\mathbb R}^n$ for $n\geq 2$,  that is to say,
\[
S^{\lambda}_{\ast}(H) f(x):=\sup_{R>0} |S^{\lambda}_R(H) f(x)|,
\]
and  they showed  that the operator $S_{\ast}^{\lambda}(H)$ is bounded on $L^p(\RN)$ whenever
\begin{eqnarray}\label{e1.6}
 p \ge {2n\over n-2}\ \ \ {\rm and }\ \ \  \lambda> \lambda(p).
\end{eqnarray}
As a consequence, we have
\begin{equation}
\label{eq:hae}
\lim\limits_{R\to \infty} S_R^{\lambda}(H) f=f\quad \text{a.e.}
\end{equation} for $f\in L^p(\RN)$ and  $p$, $\lambda$ satisfying \eqref{e1.6}.
For more results regarding the Hermite expansion \eqref{e1.3}, the estimate for the Hermite spectral projection,  and
 the Bochner-Riesz means for the Hermite operator, we refer the reader to \cite{ Kar1,   Th22, Th1, Th4,  KT,  DOS,  COSY, JLR}
 and references therein.

The following is the main result of this paper which establishes a.e convergence of the operator $S^{\lambda}_{R} (H)$ up to
 the sharp summability index for $p\ge 2$.

\begin{thm}\label{th1.1}  Let $2\le p< \infty$ and $\lambda\ge 0$.  Then, for any   $f\in L^p(\mathbb R^n)$ we have  \eqref{eq:hae} whenever $
 \lambda>   \lambda(p)/2$.  In particular, for $n=1$,
	\eqref{eq:hae} holds 	for all $f\in L^p({\mathbb R})$
	whenever $\lambda>0$. Conversely,  if    \eqref{eq:hae} holds  for all $f\in L^p(\mathbb R^n)$ with $n\ge 2$ and $ 2n/(n-1)<p<\infty$,
	 we  have   $\lambda\geq  \lambda(p)/2$.
\end{thm}

Except the endpoint cases
  Theorem \ref{th1.1} almost completely settles
the a.e. convergence problem of $S_R^\lambda(H)f$
as $R\to \infty$ with $f\in L^p(\mathbb R^n)$, $p
\ge 2$.
As is already mentioned, $S_R^\lambda(H)$ converges in $L^p$ only if $\lambda>\lambda(p)$. Surprisingly,
 we only need the half of the critical summability index  $\lambda(p)$  in order to guarantee a.e. convergence of  $S_R^\lambda(H)f$.
Unlike the classical Bochner-Riesz means the critical indices for  $L^p$ convergence and a.e. convergence for Hermite operators do not match.

Let us now recall from \cite[pp.320-321]{CSo} (also  \cite{LS}) how the sharpness of the result in \cite{CRV} can be justified
for the classical Bochner-Riesz means $S_R^\lambda$.
  In order to consider  a.e. convergence
of $S_R^\lambda f$ with  $f\in L^p$,  $S_R^\lambda f$ should be defined at least as a tempered distribution for $f\in L^p$. If so,
 by duality  $S^\lambda$ is defined from Schwartz class ${\mathscr S}$ to $L^{p’}$. This implies the convolution kernel $K^\lambda$
 of $S^\lambda$  is in $L^{p'}$, so it follows that $\lambda>\lambda(p)$
because  $K^\lambda\in L^{p'}$ if and only if  $\lambda>\lambda(p)$.
However, this kind of argument does not work for the Bochner-Riesz means for the Hermite operator since $S_R^\lambda(H) f$
 is well defined for any $f\in L^p$.
To show the necessity part of Theorem \ref{th1.1}  we make use of the Nikishin-Maurey theorem
by which a.e. convergence implies a weighted inequality for the maximal operator
$S^{\lambda}_{\ast} (H)$.  We show such a maximal estimate can not be true if $\lambda<\lambda(p)/2$.  See Proposition  \ref{nec-1} below.

The sufficiency part of  Theorem \ref{th1.1}  relies on the maximal estimate which is a typical device in
the study of almost everywhere convergence. In order to show
\eqref{eq:hae}  we  consider the corresponding maximal operator $S_*^\lambda(H) $ and prove the following weighted estimate, from which
we deduce a.e. convergence of  $S_R^{\lambda}(H) f$ via the standard argument.

\begin{thm}\label{th1.2}  Let
  $0\leq \alpha<n$. The operator  $S_*^\lambda(H)$ is bounded on $L^2(\RN, (1+|x|)^{-\alpha})$ if
\[\lambda>\max\Big\{\frac{\alpha-1}{4},0\Big\}. \]
Conversely, if
  $S_*^\lambda(H)$ is bounded on $L^2(\RN, (1+|x|)^{-\alpha})$,   then   $\la \ge\max\big\{ {(\alpha-1)}/{4},0\big\}$.
\end{thm}

Once we have   Theorem~\ref{th1.2} it is easy to deduce  the sufficiency part of     Theorem~\ref{th1.1}.
  Indeed,  via a standard approximation argument
 (see, for example, \cite{SW} and \cite[Theorem 2]{Th1})  Theorem~\ref{th1.2} establishes   \ale  convergence of $S_R^\lambda(H)f$
   for all $f\in L^2({\mathbb R^n}, \, (1+|x|)^{-\alpha})$ provided that
 $
\lambda>\max\big\{ {(\alpha-1)}/{4},0\big\}.
$
Now,   for given $p\geq 2$ and $\lambda>\lambda(p)/2$ we can choose an $\alpha$ such that
$ \alpha> n(1-2/p)$ and $
\lambda>\max\big\{ {(\alpha-1)}/{4},0\big\}.
$
Our choice of $\alpha$  ensures that  $f\in L^2({\mathbb R^n}, \, (1+|x|)^{-\alpha})$  if $f\in L^p$ as it follows by H\"older's inequality.
Therefore, this yields a.e. convergence of $S_R^\lambda(H)f$   for  $f\in L^p(\RN)$  if $\lambda> \lambda(p)/2$.

The use of weighted $L^2$ estimate in the study of pointwise convergence for Bochner-Riesz means goes back to Carbery et al. \cite{CRV}.
It turned out that the  same strategy is also efficient for similar problems  in different settings.
For example, see  \cite{A, LS} for a.e. convergence of the classical
 Bochner-Riesz means at the critical index $\lambda(p)$ with $p> 2n/(n-1)$  and  see  \cite{GM, HM}  for  a.e. convergence for the Bochner-Riesz
 means associated with the sub-Laplacian on the Heisenberg  group.

\subsection*{Square function estimate on weighted $L^2$-space} The proof of the sufficiency part of Theorem~\ref{th1.2} relies on a weighted
$L^2$-estimate for the square function $\mathfrak S_\delta$ which is defined by
\begin{equation}\label{e4.6}
 \mathfrak S_{\delta}f(x)=\left( \int_0^{\infty} \Big|\phi \Big(\delta^{-1}\Big(1-{ {H}\over t^2} \Big) \Big)f(x)\Big|^2
   {dt\over t}\right)^{1/2},\quad 0<\de<1/2,
  \end{equation}
  where $\phi$ is a fixed $C^{\infty}$ function supported in $[2^{-3}, 2^{-1}]$ with $ |\phi|\leq 1$.\footnote{Here,  for any bounded function $\mathfrak M$
  the operator $\mathfrak M(H)$ is  defined by $\mathfrak M(H)= \sum\limits_{k=0}^{\infty}\mathfrak M(2k+n)P_k$.} The following is our main estimate.

\begin{prop}\label{prop4.1}
Let  $0<\delta\leq 1/2$, $0<\epsilon \leq 1/2$,  and let $0\leq \alpha <n$. Then, there exists a constant $C>0$, independent of $\de$ and $f$, such that
 \begin{equation}\label{e4.77}
  \int_{\RR^n} |\mathfrak S_{\delta}f(x)|^2(1+|x|)^{-\alpha}dx\leq C\delta A_{\alpha, n}^{\epsilon}(\delta)
  \int_{\RR^n} |f(x)|^2(1+|x|)^{-\alpha}dx,
  \end{equation}
  where
   \begin{eqnarray}\label{bbb}
  A_{\alpha, n}^{\epsilon}(\delta): =
 \left\{   \begin{array}{llll}   \,\,\delta^{-\epsilon}, \ & 0\leq\alpha\le 1,   & \text{ if } n=1,     \\[6pt]
 \delta^{{1\over 2}-{\alpha\over 2}}, \ & 1<\alpha<n, & \text{ if } n\ge 2.    \end{array}     \right.
   \end{eqnarray}
  \end{prop}

A similar  estimate with the homogeneous weight $|x|^{-\alpha}$ was obtained by Carbery et al.  \cite{CRV}  for the
square function associated to the Laplacian $\Delta$:
\[S_{\delta}f(x):= \Big( \int_0^{\infty} \Big|\phi \Big(\delta^{-1}\Big(1+{t^{-2}{\Delta}} \Big) \Big)f(x)\Big|^2
   \frac{dt}t\Big)^{1/2}.\]
Though we make use of the weighted $L^2$ estimate as in \cite{CRV} there are notable differences which are due to
special properties of the Hermite operator and   they eventually
lead to improvement of the summability indices.  Let $\mathcal P_k$ be the Littlewood-Paley projection operator which is given by
$\widehat{\mathcal P_k f}(\xi)=\phi(2^{-k}|\xi|) \widehat f(\xi)$ for $\phi\in C_c^\infty(2^{-1}, 2)$. Thanks to the
scaling property of the Laplacian, the estimate for
$S_{\delta}(\mathcal P_kf) $ can  be reduced to the equivalent estimate for $S_{\delta}(\mathcal P_0f) $. This tells that
contributions from different dyadic frequency pieces are basically identical.  However, this is not the case for $\mathfrak S_{\delta}f$.
As for the Hermite case estimate \eqref{e4.77},  the high and low frequency parts exhibit considerably different natures.
Unlike the classical Bochner-Riesz operator,
we need to handle them separately.

\subsection*{Basic estimates}
 As is to be seen in Section~\ref{sec3} below, the proof of  Proposition~\ref{prop4.1} mainly  depends on the following two lemmas.

 \begin{lem}\label{prop2.1} Let $\alpha\geq 0$. Then,  the estimate
\begin{eqnarray}\label{e1.8}
\|(1+|x|)^{2\alpha}f\|_2\leq C\|(1+H)^{\alpha}f\|_2
\end{eqnarray}  holds for any $f\in {\mathscr S}(\RN)$. Here,
 ${\mathscr S}(\RN)$ stands for  the class of Schwartz functions in $\RN.$
 \end{lem}
Clearly, this can not be true if $H$ is replaced by $-\Delta$. It should be noted that
the estimate \eqref{e1.8} becomes more efficient when we deal with the low frequency part of the function.
 The second is  a type of trace lemma (Lemma~\ref{le3.1}) for the Hermite operator.
   In fact, we obtain

  \begin{lemma}\label{le3.1}
 For  $    \alpha>1$, there exists a constant $C>0$ such that  the estimate
  \begin{eqnarray}\label{e1.9}
\|\chi_{[k,k+1)}(H)\|_{L^2(\mathbb R^n)\to L^2(\RN, \,  (1+|x|)^{-\alpha})} \leq Ck^{-\frac{1}{4}}
\end{eqnarray}
   holds for every $k\in \NN$.
\end{lemma}

In our proof of \eqref{e4.77}  this inequality  \eqref{e1.9} takes the place of
the classical trace lemma  which  was the main tool in  \cite{CRV}. The  trace lemma tells that
 a function in the   Sobolev space ${\dot W}^{\alpha, 2}(\RN)$ can be
 restricted to ${\mathbb S}^{n-1}$ as an $L^2$ function. By taking Fourier transform and Plancherel's theorem,
  this  can be equivalently formulated as follows:
\[
\int_{\RN}\big|\chi_{[1-\epsilon, 1+\epsilon ]} ( \sqrt{-\Delta} )f(x)\big|^2 dx
 \leq C\epsilon
   \int_{\RR^n} |f(x)|^2|x|^\alpha dx.
\]
In contrast with the case of the Laplacian where the  trace inequality should take a scaling-invariant form, that is
 to say, the weight  should be homogeneous, we have  the inhomogeneous weight  $(1+|x|)^{-\alpha}$
 in both of the estimates \eqref{e1.8} and \eqref{e1.9}.  As to be seen later, this is  related  to  the fact that the spectrum of
 the Hermite operator $H$ is bounded away from the origin.

We  show Proposition~\ref{prop4.1} by making use of both  of the estimates \eqref{e1.8} and \eqref{e1.9}.
The proof of Proposition~\ref{prop4.1} divides into two parts depending on size of frequency in the spectral decomposition \eqref{e1.5}.
For the high frequency part ($k \gtrsim \delta^{-1}$ in  \eqref{e1.5})
the key tool is the estimate \eqref{e1.9}, which we combine with spatial localization argument based on the finite
 speed of propagation of the wave operator
$\cos(t\!\sqrt H)$.  The estimate \eqref{e1.9} can be compared  with the restriction-type estimate  due to
Karadzhov \cite{Kar}:
\begin{eqnarray}\label{e6.1}
\|\chi_{[k,k+1)}(H)\|_{2\to p} \leq Ck^{\frac{n}{2}(\frac{1}{2}-\frac{1}{p})-\frac{1}{2}},  \ \ \ \  \ \ \forall k\geq 1.
\end{eqnarray}
 The bound  in \eqref{e1.9}
 is much smaller  than  that in \eqref{e6.1} when $k$ is large.  So, the estimate  \eqref{e1.9} becomes more efficient
 in the high frequency regime.  In fact,   the estimate \eqref{e6.1} was  used to  show the sharp $L^p$--bounds on
 $\mathfrak S_\delta$ for $ {2n}/{(n-2)}\leq p\leq \infty,$ $n\ge 2$, \cite[Proposition 5.6]{CLSY}.
 In the low frequency part ($k \lesssim \delta^{-1}$ in  \eqref{e1.5}), inspired by \cite[Lemma 5.7]{CLSY},
  we directly obtain the estimate  using the  estimate \eqref{e1.8}.
 The estimate \eqref{e1.8} does not seem  to be so efficient since  the bound gets worse as the  frequency increases,
 but it is remarkable that this bound is good enough to yield the sharp result in Theorem~\ref{th1.2} via
 balancing the estimates for  low and high frequencies (see Remark \ref{re:aboutdelta}).

  \subsection*{Organization of the paper} The rest of the paper is organized as follows.
 In Section 2 we prove Lemma~\ref{prop2.1},  Lemma~\ref{le3.1} and the Littlewood-Paley
  inequality for the  Hermite operator, which provide basic estimates required
 for the proof of  Proposition \ref{prop4.1}.   We give  the proof  of  the sufficiency part of
   Theorem~\ref{th1.2} in Section \ref{sec3} by establishing
 the square function estimate in Proposition \ref{prop4.1}. In Section \ref{sec4} we show the
 sharpness of the summability indices, hence we  complete the proofs  of
 Theorem~\ref{th1.1} and Theorem~\ref{th1.2}.

\medskip

\section{Some weighted estimates for the Hermite operator }
 \setcounter{equation}{0}
 In this section, we prove Lemma \ref{prop2.1}, Lemma  \ref{le3.1} and the Littlewood-Paley inequality  for  the Hermite operator
 in
${\mathbb R}^n$,  which are  to be used  in the proof of the sufficiency part of Theorem~\ref{th1.2} in Section 3.

\subsection{Proof of  Lemma~\ref{prop2.1} }  In order to show Lemma~\ref{prop2.1},
we use the following Lemmas~\ref{le2.2}, \ref{le2.3} and  \ref{le2.4}.

\begin{lemma}\label{le2.2} For all $\phi\in {\mathscr S}(\mathbb R^n)$, we have
$
\|\phi\|_2\leq \|H\phi\|_2
$
and   $
\|H^k\phi\|_2\leq \|H^{k+m}\phi\|_2
$  for any $k,m\in \NN$.
\end{lemma}

\begin{proof}
This follows from the fact  that  the first eigenvalue of $H$ is bigger than or equal to $1$.
\end{proof}

\begin{lemma}\label{le2.3}
Let $n=1$. Then, for all $\phi\in {\mathscr S}(\mathbb R)$,
$$
\|x^2\phi\|^2_2+\Big\|\frac{d^2}{d x^2}\phi\Big\|_2^2+\Big\|x\frac{d}{dx}\phi\Big\|_2^2\leq3 \|H\phi\|^2_2.
$$
\end{lemma}
\begin{proof}
Since $\|H\phi\|^2_2= \langle (-\frac{d^2}{d x^2}+x^2)\phi,(-\frac{d^2}{d x^2}+x^2)\phi\rangle$,
a simple calculation shows
\begin{eqnarray*}
\|H\phi\|^2_2= \Big\|\frac{d^2}{d x^2}\phi\Big\|_2^2+2{\rm Re}
\Big\langle \frac{d}{d x}\phi,2x\phi\Big\rangle+2 \Big\|x\frac{d}{dx}\phi\Big\|_2^2+\|x^2\phi\|^2_2.
\end{eqnarray*}
We now observe
$
2{\rm Re}\langle \frac{d}{d x}\phi,2x\phi\rangle=\left\langle \frac{d}{d x}\phi,2x\phi\right\rangle+\langle 2x\phi,\frac{d}{d x}\phi\rangle
=-2\langle \phi, \phi\rangle.
$
This and the above give
\begin{eqnarray*}
\|x^2\phi\|^2_2+\Big\|\frac{d^2}{d x^2}\phi\Big\|_2^2+2 \Big\|x\frac{d}{dx}\phi\Big\|_2^2&=&\|H\phi\|^2_2+2\|\phi\|_2^2\leq 3\|H\phi\|^2_2
\end{eqnarray*}
as desired.  For the last inequality we use Lemma \ref{le2.2}.
\end{proof}

\begin{lemma}\label{le2.4} Let $n=1$. Then, for $\phi\in {\mathscr S}(\mathbb R)$ and for $k\in \NN$,  we have
\begin{eqnarray}\label{e2.2}
\|x^{2k}\phi\|_2\leq C_k\|H^k\phi\|_2
\ \ \ {\rm and} \ \ \
\|Hx^{2(k-1)}\phi\|_2\leq D_k\|H^k\phi\|_2.
\end{eqnarray}
\end{lemma}

\begin{proof}  We begin with noting that,  if $k=1$,
  the first estimate in  \eqref{e2.2} holds  with $C_1=\sqrt 3$ by Lemma~\ref{le2.3}, and the second with $D_1=1$.
  We now proceed to prove   \eqref{e2.2} for   $k\geq 2$  by induction.  Assume that
  \eqref{e2.2} holds for $k-1$ with some constants $C_{k-1}$ and $D_{k-1}$. A computation  gives
\begin{equation}
\label{e2.3}
\begin{aligned}
 H(x^{2(k-1)}\phi)&=-[(2k-2)(2k-3)-2(2k-2)(2k-4)]x^{2(k-2)}\phi \\
 & \qquad \quad -2(2k-2)x\frac{d}{dx}\big(x^{2(k-2)}\phi\big)+x^{2k-2}H\phi.
\end{aligned}
\end{equation}
By \eqref{e2.3},  Lemma~\ref{le2.3}, and our induction assumption we see that
\begin{eqnarray*}
& &  \|H(x^{2(k-1)}\phi)\|_2
 \\
&\leq& (2k-2)(3k-5)\|x^{2(k-2)}\phi\|_2
 +
2(2k-2)\|x\frac{d}{dx}(x^{2(k-2)}\phi)\|_2
 + \|x^{2k-2}H\phi\|_2\\
&\leq &  (2k-2)(3k-5)   C_{k-2}\|H^{k-2}\phi\|_2
 +2(2k-2)C_1\|H(x^{2(k-2)}\phi)\|_2+C_{k-1}\|H^k\phi\|_2\\
&\leq &   (2k-2)(3k-5)    C_{k-2}\|H^{k-2}\phi\|_2
 +2(2k-2)C_1D_{k-1}\|H^{k-1}\phi\|_2+C_{k-1}\|H^k\phi\|_2.
\end{eqnarray*}
Hence, we get the estimate
\[\|H(x^{2(k-1)}\phi)\|_2\leq D_k\|H^k\phi\|_2\] with
$D_k=    (2k-2)(3k-5)   C_{k-2}
 +4(2k-2)C_1D_{k-1}+C_{k-1}.
 $
On other hand, we also have
\begin{eqnarray*}
\|x^{2k}\phi\|_2\leq C_1\|H(x^{2(k-1)}\phi)\|_2\leq C_k \|H^k \phi\|_2\ \ \
\end{eqnarray*}
with $C_k=C_1D_k, k\geq 2$. So, we readily get the estimates in \eqref{e2.2}. This completes the proof.
\end{proof}

Now we are ready to prove Lemma~\ref{prop2.1}.

\begin{proof}[Proof of Lemma ~\ref{prop2.1}]
Define $H_i=-\frac{\partial^2}{\partial x_i^2}+x_i^2$, $i=1,2,\ldots,n$.
By Lemma~\ref{le2.4}
$$
\||x_i|^{2k} f\|_2\leq C_k \|H_i^{k} f\|_2
$$
for all positive natural numbers $k\in \NN$.
Hence, we have
$$
\|(1+|x|)^{2k}f\|_2^2
 \leq C_k\big(\|f\|_2^2+ \sum_{i=1}^{n}\| |x_i|^{2k} f\|_2^2\big)
 \leq C_k \big(\|f\|_2^2+ \sum_{i=1}^{n}
\|H_{i}^{k} f\|_2^2\big).
$$
Since all $H_i$ are non-negative selfadjoint operators and
commute strongly (that is, their spectral resolutions commute),
 the operators
$\prod_{i=1}^nH_{i}^{\ell_i}$ are non-negative selfadjoint  for all $\ell_i\in \mathbb{Z}_+$.
Hence
$$
1+\sum_{i=1}^{n}H_{i}^{2k}\leq(1+\sum_{i=1}^{n}H_{i})^{2k}   = (1+H)^{2k}
$$
for all $k\in \NN$. Combining this with the above inequality  we get
\begin{eqnarray*}
\|(1+|x|)^{2k}f\|_2^2
\leq
C_k\Big \langle\Big( 1+\sum_{i=1}^{n}H_{i}\Big)^{2k}f,\, f\Big\rangle =C_k
\|(1+H)^kf\|^2_2.
\end{eqnarray*}
This proves estimates \eqref{e1.8} for all $\alpha \in \NN$.
Now, by  virtue of L\"owner-Heinz inequality (see, e.g., \cite[Section I.5]{Co})
we can extend this estimate to all $\alpha\in
[0,\infty)$. This completes the proof of  Lemma~\ref{prop2.1}.
\end{proof}

\subsection{  The proof of   Lemma  \ref{le3.1}: Trace lemma for the Hermite operator}
Our proof of  the estimate  \eqref{e1.9}    is  inspired by  the argument in \cite[Theorem 3.3]{BR} where
the authors  obtained local smoothing estimate for the Hermite Schr\"odinger propagator.

\begin{proof}
To show  \eqref{e1.9}, it is sufficient to show
\begin{eqnarray}\label{e3.12dual}
\int_{[-M, M]^n}|\chi_{[k,k+1)}(H)f(x) |^2   dx
\leq  C  M k^{-\frac{1}{2}}\|f\|_2^2
\end{eqnarray}
for every $M\ge 1$. Indeed, the estimate \eqref{e1.9} immediately follows by
  decomposing $\mathbb R^n$ into dyadic shells and applying  \eqref{e3.12dual} to each of them because  $\alpha>1$.

Let us prove \eqref{e3.12dual}. For every $f\in{\mathscr S}(\RN)$,
 we may write  its Hermite expansion
  $f(x)=\sum_{\mu}\langle f, \Phi_{\mu}\rangle \Phi_\mu(x)$ as in \eqref{e1.3}.
 Considering this spectral decomposition, clearly  we may decompose
 \Be \label{decomp}  f=\sum\limits_{i=1}^n f_i \Ee
 such that $f_1,\dots, f_n$ are orthogonal to each other and, for $1\le i\le n$,
 $\mu_i\geq |\mu|/n$ whenever $\langle  f_i, \Phi_\mu \rangle\neq 0$ (see for example, \cite{BR}).
 Recalling   that the Hermite functions $\Phi_{\mu}$    are  eigenfunctions  for
the Hermite operator $H$, it is clear that
$$
\chi_{[k,k+1)}(H)f_i(x)=\sum\limits_{2|\mu|+n=k}\langle  f_i, \Phi_\mu \rangle  \Phi_\mu(x).
$$
Note that   $\mu_i\sim |\mu|$ if  $\langle  f_i, \Phi_\mu \rangle\neq 0$, so in order to show \eqref{e3.12dual}  it is enough to show that
 \Be \label{e3.12dual1}
\int_{[-M, M]^n}|\chi_{[k,k+1)}(H)f_i(x) |^2   dx
\leq  C  M \sum_{{2|\mu|+n=k}}\mu_i^{-\frac{1}{2}}|\langle  f_i, \Phi_\mu \rangle|^2
\Ee
for each $i=1,\dots, n$ and $M>0$.
By symmetry we have only to show \eqref{e3.12dual1} with $i=1$.
For the purpose  we do not need the particular structure of $f_1$, so
let us set  $g:=f_1$ for a simpler notation.

 Let us write  $g(x)=\sum_{\mu} c(\mu) \Phi_\mu(x)$ with $c(\mu)=\langle g, \Phi_{\mu}\rangle.$  Hence, we have
\[
|\chi_{[k,k+1)}(H)g(x)|^2=\sum_{2|\mu|+n=k}\sum_{2|\nu|+n=k} c(\mu) \overline{c(\nu)}   \prod_{i=1}^n  h_{\mu_i}(x_i) {h_{\nu_i}(x_i)}.
\]
Using this,  by Fubini's theorem  it follows that
\begin{eqnarray*}
 & &\hspace{-1cm}  \int_{[-M, M]^n}|\chi_{[k,k+1)}(H)g(x) |^2   dx
 \\
&\le & \sum_{2|\mu|+n=k}\sum_{2|\nu|+n=k} c(\mu) \overline{c(\nu)}\int_{-M}^Mh_{\mu_1}(x_1){h_{\nu_1}(x_1)}dx_1
  \prod_{i=2}^n \langle h_{\mu_i}, h_{\nu_i}\rangle.
 \end{eqnarray*}
Since $h_{\mu_i}$ are orthogonal to each other,
we have $\mu_i=\nu_i$ for $i=2,\ldots,n$  whenever $
\langle h_{\mu_i}, h_{\nu_i}\rangle\neq 0
$
 and we also have
$\mu_1=\nu_1$ since  $2|\mu|+n=k=2|\nu|+n$.  Thus,
\[
\int_{[-M, M]^n}|\chi_{[k,k+1)}(H)g(x) |^2   dx
\leq \sum_{2|\mu|+n=k} |c(\mu) |^2 \int_{-M}^M  h^2_{\mu_1}(x_1)dx_1 .
\]
Therefore, to complete the proof  it suffices to show that
\[\int_{-M}^M  h^2_{\mu_1}(t)dt\le CM\mu_1^{-1/2}.\]
If $\mu_1\leq M^2$, the estimate is trivial because $\|h_{\mu_1}\|_2=1$. Hence, we may assume $\mu_1> M^2$.
By the property of the Hermite functions (see \cite[Lemma 1.5.1]{Th3})
there exists a constant $C>0$ such that
$
|h_{\mu_1}(t)|\leq C\mu_1^{-1/4}$ provided that  $ t\in [-M, M]$ and $\mu_1>M^2$.  Thus, we get the desired estimate, which
completes the proof of Lemma~\ref{le3.1}.
\end{proof}

\subsection{An extension of   the estimate \eqref{e1.9}}
We modify the estimate \eqref{e1.9}  into  a form which is suitable for our purpose.
 For any function $F$ with support in $[0, 1]$ and $2\leq q<\infty$,  we  define
 \begin{eqnarray}\label{e3.1}
\| F\|_{N^2,q}:= \left(\frac{1}{N^2}\sum_{\ell = 1 }^{N^2}
\sup_{\lambda \in [\frac{\ell -1}{N^2}, \frac{\ell}{N^2} ) }
|F(\lambda)|^q\right)^{1/q}, \ \ \ \ N\in{\mathbb N}.
\end{eqnarray}
For $q=\infty$, we put $\|F\|_{N^2,\infty}=\|F\|_\infty$ (see \cite{COSY, CowS,  DOS}).
Then we have the following result which is a generalization of Lemma \ref{le3.1}.

\begin{lemma}
\label{le3.2}
For  $    \alpha>1$ we have
\begin{eqnarray*}
\int_{\RN}|F(\!\SH\,)f(x) |^2  (1+|x|)^{-\alpha} dx  \leq CN \|\delta_N F\|_{N^2,2}^2  \int_{\RR^n} |f(x)|^2 dx
\end{eqnarray*}
for any function $F$ with support in $[N/4, N]$ and $N\in {\mathbb N}$, where $\de_N F(\la)$ is defined by $F(N\la)$.
\end{lemma}

\begin{proof}  Since the operator $F(\!\SH\,)$ is selfadjoint, it is sufficient to show the dual estimate
\Be\label{e3.2}
\int_{\RN}|F(\!\SH\,)f(x) |^2 dx  \leq CN \|\delta_N F\|_{N^2,2}^2  \int_{\RR^n} |f(x)|^2  (1+|x|)^{\alpha}  dx.
\Ee
By orthogonality
\[
\int_{\RN}|F(\!\SH\,)f(x) |^2 dx  \,\le   \sum_{\ell = N^2/16}^{N^2} \Big \| \chi_{\big[{\ell-1\over N} ,{\ell\over N}\big)}(\!\SH\,)F(\!\SH\,)f \Big\|_{2}^2
\]
because $F$ is  supported in $[N/4, N]$. Note  that
\[ \|\chi_{[{(\ell-1) \over N } ,{\ell \over N })}
(\!\sqrt H\,)f \|_{2}^2   \le \|\chi_{\big[{(\ell-1)^2\over N^2} ,{(\ell-1)^2\over N^2} +2\big)}
(H)f \|_{2}^2\,.\] Hence, it follows that
\[
\int_{\RN}|F(\!\SH\,)f(x) |^2 dx \leq  \sum_{\ell = N^2/16}^{N^2}  \sup_{\,\lambda \in \big[{\ell-1\over N} ,{\ell\over N}\big) } |F( \lambda)|^2
\Big\|\chi_{\big[{(\ell-1)^2\over N^2} ,{(\ell-1)^2\over N^2} +2\big)}
(H)f \Big\|_{2}^2.
\]
Since $\frac{\ell-1}{N}\sim N$, applying  \eqref{e1.9} we  obtain
\[
\int_{\RN}|F(\!\SH\,)f(x) |^2 dx
\le  \frac CN \sum_{\ell = N^2/16}^{N^2}   \sup_{\lambda \in \big[{\ell-1\over N} ,{\ell\over N}\big) } |F( \lambda)|^2
 \int_{\RR^n} |f(x)|^2 (1+|x|)^\alpha dx.
\]
Thus, the estimate~\eqref{e3.2} follows from \eqref{e3.1}.
\end{proof}

In Lemma~\ref{le3.1}, the estimate \eqref{e1.9} is established for    $\alpha>1$  (see also \cite[Theorem 3.3]{BR}).
In what follows, we  use a bilinear interpolation theorem to extend the range of $\alpha$
to    $0<\alpha\leq 1$.
We recall that $[\cdot, \cdot]_{\theta}$ stands for the complex interpolation bracket (for example, see \cite{Lof, Cal}).

\begin{lemma}\label{le:interpolation}
Let $(A_i, B_i), i=1, 2$ and $(A,B)$ be interpolation pairs. Suppose
$T$ is a bilinear operator defined on $\oplus_{i=1}^2(A_i\cap B_i)$ with values in $A\cap B$  such that
$$
\|T(x_1, x_2)\|_{A}\leq M_0 \prod_{i=1}^2\|x_i\|_{A_i}, \quad
\|T(x_1, x_2)\|_{B}\leq M_1 \prod_{i=1}^2\|x_i\|_{B_i}.
$$
Then, for  $\theta\in [0,1]$, we have
$$
\|T(x_1, x_2)\|_{[A, B]_\theta}\leq M_0^{1-\theta}M_1^\theta\prod_{i=1}^2\|x_i\|_{[A_i, B_i]_\theta}.
$$
Thus $T$ can be extended continuously from  $ \oplus_{i=1}^2 [A_i, B_i]_\theta$ into $[A, B]_\theta$ for any $\theta\in [0,1]$.
\end{lemma}

For the proof of Lemma~\ref{le:interpolation}, we refer the reader  to  \cite[p.118]{Cal}.
Making use of Lemma \ref{le:interpolation}, we obtain  the following result.

\begin{lemma}
\label{le33.1}
Let  $  0<  \alpha \leq 1$, $N\in {\mathbb N}$, and let $F$ be  a function supported  in $[N/4, N]$. Then, for  any $\varepsilon>0$,  we have
\begin{eqnarray}\label{bbbb}\hspace{0.6cm}
\int_{\RN}|F(\!\SH\,)f(x) |^2 (1+|x|)^{-\alpha} dx  \leq C_\varepsilon N^{\frac{\alpha}{1+\varepsilon}}
\|\delta_N F\|_{N^2,q}^2  \int_{\RR^n} |f(x)|^2   dx \bf
\end{eqnarray}
with $q={2\alpha^{-1}(1+\varepsilon)}$ for some constant $C_\varepsilon>0$ independent of $f$ and $F$.
\end{lemma}

\begin{proof}
Clearly, \eqref{bbbb} holds for $N=1$. Now
we fix $N\in{\mathbb N}$ and $N\geq 2$,  and let  $\mathcal A$ be  a collection  of functions supported in  $[N/4,N]$  which is given by
$$
\mathcal{A}:=\Big\{G\in L_{loc}^1: G(x)=a_0\chi_{[\frac{N}{4},\frac{[N^2/4]+1}{N})}(x)
+\sum_{i=[N^2/4]+2}^{N^2} a_i\chi_{[\frac{i-1}{N},\frac{i}{N})}(x), \  a_0,a_i\in \CC\Big\}.
$$
Then, for any Borel set $Q$,  we define a normalized counting measure $\nu$ on $\RR$ by setting
$$
\nu(Q):=\frac{1}{N^2-[N^2/4]-1}\left( \#\Big\{ i: \frac{2i-1}{2N}\in Q, i=[N^2/4]+2,\ldots,N^2\Big\} + C_{NQ}\right),
$$
where $N\geq 2$ and
\begin{eqnarray*}
C_{NQ}=
\left\{
\begin{array}{lll}
1\ \ \ &{\rm if}\ \ \ N/4\in Q;\\[6pt]
0\ \ \ &{\rm if}\ \ \ N/4\not\in Q.
\end{array}
\right.
\end{eqnarray*}
 We also define an $L^q$ norm on $\mathcal{A}$ by
$$
\|G\|_{L^q(d\nu)}:=\Big(\int_{N/4}^{N} |G(x)|^q d\nu(x)\Big)^{1/q}.
$$
Hence,   $
\|G\|_{L^q(d\nu)}\sim \|\delta_N G\|_{N^2,q},
 $
and the space  $\mathcal{A}$ equipped with this norm becomes a Banach space which is denoted by $\mathcal{A}_q$.
It also follows (for example, see \cite{Lof}) that
$$
[\mathcal{A}_2,\mathcal{A}_\infty]_s=\mathcal{A}_q, \quad  {1\over q}= {1-s\over 2},  \quad s\in [0,1].
$$
Let us denote by  $ \mathcal{B}_\alpha$ the space
\[
\mathcal{B}_\alpha:=\big\{f\in L_{loc}^1(\RN): \int_{\RN} |f(x)|^2(1+|x|)^\alpha dx<\infty\big\}
\]
equipped with the norm
$
\|f\|_{\mathcal{B}_\alpha}:=\|f\|_{L^2(\RN, \,  (1+|x|)^{-\alpha})}$.
Thus we have
$
[\mathcal{B}_\nu,\mathcal{B}_0]_{\theta}=\mathcal{B}_s, s=(1-\theta)\nu.
$
We consider the bilinear  operator $T$ given by
$$
T(G,f):=G(\!\sqrt{H}\,)f.
$$
By Lemma~\ref{le3.2} and duality, for any $\varepsilon>0$ we have
$$
\|T(G, f)\|_{L^2(\RN)}\leq CN^{1/2}\|G\|_{\mathcal{A}_2}\|f\|_{\mathcal{B}_{1+\varepsilon}}.
$$
Since
$
\int_{\RN}|G(\!\SH\,)f(x) |^2 dx  \leq \|\delta_N G\|_{\infty}^2  \int_{\RR^n} |f(x)|^2  dx,
$
we also have
 \[
\|T(G, f)\|_{L^2(\RN)}\leq \|G\|_{\mathcal{A}_\infty}\|f\|_{\mathcal{B}_0}.
 \]
Now, taking $A_1=\mathcal{A}_2$, $A_2=\mathcal{B}_{1+\varepsilon}$,
$B_1=\mathcal{A}_\infty$, $B_2=\mathcal{B}_{0}$, and $A=B=L^2(\RN)$,  we apply Lemma~\ref{le:interpolation} to get
$$
\|T(G, f)\|_{L^2(\RN)}\leq C N^{(1-\theta)/2}\|G\|_{\mathcal{A}_{2/(1-\theta)}}\|f\|_{\mathcal{B}_{(1-\theta)(1+\varepsilon)}}
$$
for $\theta\in [0,1]$. Let $(1-\theta)(1+\varepsilon)=\alpha$. Then,   equivalently,  for $G\in \mathcal{A}$  we have
\begin{equation}\label{e-special}
\int_{\RN}|G(\!\SH\,)f(x) |^2 dx \leq C_\varepsilon N^{\frac{\alpha}{1+\varepsilon}}
\|\delta_N G\|_{N^2,2(1+\varepsilon)/\alpha}^2  \int_{\RR^n} |f(x)|^2 (1+|x|)^\alpha dx.
\end{equation}

We now extend the estimate  \eqref{e-special}
to general functions supported in $[N/4,N]$.  For a function $F$ supported in $[N/4,N]$,  set
$a_0=\sup_{\lambda\in [\frac{N}{4},\frac{[N^2/4]+1}{N})}|F(\lambda)|$ and
 $
a_i=\sup_{\lambda\in [\frac{i-1}{N},\frac{i}{N})}|F(\lambda)|, i=[N^2/4]+2,\ldots,N^2
$
and define $G\in \mathcal{A}$ by
$$
G(x)=a_0\chi_{[\frac{N}{4},\frac{[N^2/4]+1}{N})}(x)+\sum_{i=[N^2/4]+2}^{N^2} a_i\chi_{[\frac{i-1}{N},\frac{i}{N})}(x).
$$
Then,   we clearly have $
\|\delta_N F\|_{N^2,q}=\|\delta_N G\|_{N^2,q}
$
and
$
|F(x)|\leq |G(x)|
$ for $x\in \RR$.
Since $\langle |F|^2(\!\SH\,)f,f \rangle=\sum\limits_{k\in 2\mathbb N_0+n}|F|^2(\!\sqrt{k})\sum\limits_{n+2|\mu|=k} |\langle f, \Phi_\mu \rangle|^2$,
we have
\[
\int_{\RN}|F(\!\SH\,)f(x) |^2 dx
\le   \sum_{k\in 2\mathbb N_0+n} |G|^2(\!\sqrt{k}\,)\!\sum_{n+2|\mu|=k}\! |\langle f, \Phi_\mu \rangle|^2=
\int_{\RN}|G(\SH\,)f(x) |^2 dx.
\]
Since $
\|\delta_N F\|_{N^2,q}=\|\delta_N G\|_{N^2,q}
$, using \eqref{e-special} we get
\begin{eqnarray*}
\int_{\RN}|F(\!\SH\,)f(x) |^2 dx \le  C_\varepsilon N^{\frac{\alpha}{1+\varepsilon}}
\|\delta_N F\|_{N^2,2(1+\varepsilon)/\alpha}^2  \int_{\RR^n} |f(x)|^2 (1+|x|)^\alpha dx.
\end{eqnarray*}
By duality the desired estimate follows. This completes the proof of Lemma~\ref{le33.1}.
\end{proof}

\subsection{Littlewood-Paley inequality for the Hermite operator} We now recall  a few  standard results in the theory of spectral multipliers
of non-negative selfadjoint operators (see for example, \cite{DSY, DOS}).
By the Feynman-Kac formula we have the   Gaussian upper bound on   the semigroup kernels $p_t(x,y)$ associated to
$e^{-tH}$:
\begin{equation}\label{e2.6}
0\leq p_t(x,y) \leq  (4\pi t)^{-n/2}\exp\left(-  {  |x-y|^2\over    4t } \right)
\end{equation}
for all $t>0$,  and $x,y\in \RN.$
\begin{prop}\label{prop2.5}
Fix a non-zero $C^{\infty}$ bump function  $\varphi$ on $\mathbb R$ such that
$
{\rm supp} \, \varphi \subseteq (1, 3) $. Let $\varphi_{k}(t)=\varphi(2^{-k} t)$,  $k\in \mathbb Z$, for $t>0$. Then, for any $-n<\alpha<n$,
\Be
\label{e.lwp}
\Big\|\Big(\sum_{k=-\infty}^{\infty}\big|\varphi_k(\!{\sqrt H}\,)f\big|^2\Big)^{1/2} \Big\|_{L^2({\mathbb R}^n, \, (1+|x|)^{\alpha})}
\leq C_p\|f\|_{L^2({\mathbb R}^n,\, (1+|x|)^{\alpha})}.
\Ee
\end{prop}

This can be proved by following  the standard argument, for example, see \cite[Chapter IV]{St2}. We include  a brief  proof for  convenience of  the reader.

\begin{proof}
Let us denote by  $\{r_k\}_{k\in \mathbb Z}$ the Rademacher functions.
Set
$$
F(t, \lambda):= \sum_{k=-\infty}^{\infty} r_k(t) \varphi_k(\lambda).
$$
Let $\eta$ be a nontrivial cutoff function such that $\eta\in C_c^\infty(\mathbb R_+)$.
A straightforward computation shows that  $\sup_{R>0}\|\eta F(t, R\lambda)\|_{C^\beta} \leq C_\beta$
uniformly in $t\in [0,1]$ for every integer $\beta> n/2 + 1$.
On another hand, we note that the function $|x|^{\alpha}$ belongs to the $A_2$ class if and only if $-n<\alpha <n$ (see \cite[Example 7.1.7]{Gra}).
Thus, $1+|x|^\alpha\in A_2$    and so is   $(1+|x|)^\alpha$ for $-n<\alpha <n$.
Then we may apply \cite[Theorem 3.1]{DSY} to get
\begin{eqnarray*}
\big \|F (t, {H}) f\,\big\|^2_{L^2({\mathbb R}^n, \,  (1+|x|)^{\alpha})}\leq C \big\|f\,\big\|^2_{L^2({\mathbb R}^n, \, (1+|x|)^{\alpha})}
\end{eqnarray*}
with $C>0$  uniformly in $t\in [0,1].$  Since  $ \sum_{k=-\infty}^{\infty} |\varphi_k({H})f |^2
 \cong \int_0^1\big|F (t, {H}) f\ \big|^2  dt$ by the property of the Rademacher
 functions, taking integration in $t$ on both sides of the above inequality
yields \eqref{e.lwp}.
This proves
Proposition~\ref{prop2.5}.
\end{proof}

\section{Proof of sufficiency part of  Theorem ~\ref{th1.2}}
\label{sec3}
 \setcounter{equation}{0}
  In  this section
we prove the sufficiency part of Theorem~\ref{th1.2}, that is to say,  the operator  $S_*^\lambda(H)$ is bounded on $L^2(\RN, (1+|x|)^{-\alpha})$
whenever $0\leq \alpha<n$ and
$
\lambda>\max\{\tfrac{\alpha-1}{4},0\}.
$
To do so, as mentioned before,  we make use of the square function to control
the maximal operator.

 \subsection{Reduction to square function estimate}
 We begin with recalling  the well known identity
 \begin{eqnarray*}\label{e4.1}
	\left(1-{|m|^2\over R^2}\right)^{\lambda }=C_{\lambda, \, \rho} R^{-2\lambda }\int_{|m|}^R (R^2-t^2)^{\lambda-\rho-1}t^{2\rho+1}
	\left(1-{|m|^2\over t^2}\right)^{ \rho}dt,
\end{eqnarray*}
where $\lambda>0$, $\lambda>\rho$, and
 $C_{\lambda, \, \rho}=2\Gamma(\lambda+1)/(\Gamma(\rho+1)\Gamma(\lambda-\rho))$. Using the argument in \cite[pp.278--279]{SW},
we have
\begin{eqnarray}\label{e4.2}
	S_{\ast}^{\lambda}(H)f(x)\leq  C'_{\lambda,\, \rho} \sup_{0<R<\infty}  \Big( {1\over R}\int_0^R |S_t^{\rho}(H) f(x)|^2dt\Big)^{1/2}
\end{eqnarray}
 provided that $\rho>-1/2$ and $\lambda  >\rho+1/2$.
Via dyadic decomposition, we   write $x^{\rho}_+=\sum_{k\in{\mathbb Z}} 2^{-k\rho} \phi^\rho (2^{k}x)$
 for some $\phi^\rho\in C_c^{\infty}(2^{-3}, 2^{-1})$. Thus
\begin{eqnarray*}\label{e4.3}
(1-|\xi|^2)_+^{\rho}=:
\phi_0^{\rho} (\xi)+\sum_{k=1}^{\infty} 2^{-k\rho} \phi_k^{\rho} (\xi),
 \end{eqnarray*}
 where $\phi_k^{\rho}=\phi(2^k(1-|\xi|^2)), k\geq 1$. We also note that ${\rm supp}\ \phi_0^{\rho}\subseteq \{\xi: |\xi|\leq 7\times 2^{-3}\}$ and
 $
 {\rm supp}\ \phi_k^{\rho}\subseteq \{\xi: 1-2^{ -1-k}\leq |\xi|\leq 1-2^{-k-3}\}.
$
Using  \eqref{e4.2},    for $\lambda>\rho+1/2$    we have
  \begin{eqnarray}\label{e4.4}\hspace{1cm}
S_{\ast}^{\lambda}(H)(f)
 &\leq&  C \Big(\sup_{0<R<\infty} {1\over R}\int_0^R \Big| \phi_0^{\rho}
\big({t^{-1}  \sqrt H}\big) f \Big|^2dt\Big)^{1/2} \\
&+  &
 C\sum_{k=1}^{\infty} 2^{-k\rho} \Big(\int_0^{\infty}
\Big|\phi_k^{\rho} \big({t^{-1} \sqrt H}\big) f \Big|^2{dt\over t}\Big)^{1/2}.
\nonumber \end{eqnarray}

Before we start the proof of Proposition~\ref{prop4.1}, we show that the sufficiency part of Theorem~\ref{th1.2} is a
consequence of Proposition~\ref{prop4.1}.

\begin{proof}[Proof of sufficiency part of Theorem ~\ref{th1.2}]   Since $\lambda>0$, choosing   $\eta>0$ which is
to be taken arbitrarily  small later such that
$\lambda-\eta>0$, we set  $\rho=\lambda-\frac12-\eta$.
With our choice of $\rho$ we can  use \eqref{e4.4}.
It is easy to obtain estimate for the first  term in the right hand side of \eqref{e4.4}.  Since $(1+|x|)^{-\alpha}$
is  an $A_2$  weight, by virtue of   \eqref{e2.6}  a standard argument (see for example~\cite[Lemma 3.1]{CLSY})
yields
  \begin{eqnarray*}
  \hspace{0.2cm}
 \Big\|\sup_{0<t<\infty} \Big|\phi_0^{\rho}
\big( {t^{-1}   \sqrt H}\big) f \Big| \, \Big\|_{L^2(\RN, \, (1+|x|)^{-\alpha})} \leq
C\big\| \mathscr M f  \big\|_{L^2(\RN,\, (1+|x|)^{-\alpha})},
 \end{eqnarray*}
where $\mathscr M$ is the Hardy-Littlewood maximal operator.  By the  Hardy-Littlewood maximal estimate the right hand side
 is bounded by $C\|f\|_{L^2(\RN, \, (1+|x|)^{-\alpha})}$. Hence,   in order to prove Theorem~\ref{th1.2},
 it is sufficient to handle the remaining terms in the right hand side of \eqref{e4.4}.

We first consider  $n=1$.
Using Minkowski's inequality and  the estimate  \eqref{e4.77} with $\delta=2^{-k}$,  by our choice of $\rho$ we obtain
\begin{eqnarray*}
& &\hspace{-1cm}
\Big\|\sum_{k=1}^{\infty} 2^{-k\rho} \Big(\int_0^{\infty}
\Big|\phi_k^{\rho} \big({t^{-1} \sqrt H}\big) f \Big|^2{dt\over t}\Big)^{1/2}\Big\|_{L^2({\mathbb R}, \, (1+|x|)^{-\alpha})} \\
&\leq&
	C  \sum_{k=1}^{\infty} 2^{-k(\lambda - \eta-\epsilon/2)}
	\|f\|_{L^2({\mathbb R}, \, (1+|x|)^{-\alpha})}.
	 \end{eqnarray*}
Taking $\eta $ and $ \epsilon$ small enough, the right hand side is bounded by $C\|f\|_{L^2({\mathbb R}, \, (1+|x|)^{-\alpha})}$ for any $\lambda>0$.
Hence, this gives the desired boundedness of  $S_*^\lambda(H)$ on $L^2({\mathbb R}, (1+|x|)^{-\alpha})$ for $n=1$ and $0\le \alpha<1$.

To handle the case $n\geq 2$,  we   assume $\alpha>1$ for the moment and the range of $\alpha$ is later to be extended by interpolation.
  Similarly as before, we use  \eqref{e4.77} with $\delta=2^{-k}$  to get
\begin{eqnarray*}
& &\hspace{-1cm}
\Big\|\sum_{k=1}^{\infty} 2^{-k\rho} \Big(\int_0^{\infty}
\Big|\phi_k^{\rho} \big({t^{-1}   \sqrt H}\big) f \Big|^2{dt\over t}\Big)^{1/2}\Big\|_{L^2({\mathbb R^n}, \, (1+|x|)^{-\alpha})} \\
&\leq&
	C  \sum_{k=0}^{\infty} 2^{-k(\lambda-\frac{\alpha-1}{4}-\eta)}	\|f\|_{L^2(\RN, \, (1+|x|)^{-\alpha})}.
	 \end{eqnarray*}
Thus, taking small enough $\eta$ we see that
$S_*^\lambda(H)$ is bounded  on $L^2(\RN, (1+|x|)^{-\alpha})$ for $n\ge 2$ and $1<\alpha<n$ provided that
$ \lambda>\frac{\alpha-1}4$.  On the other hand,   we note that  $S_*^\lambda(H)$ is bounded  on $L^2(\RN)$ for any $\lambda>0$,
see for example,   \cite[Corollary 3.3]{CLSY}.  Interpolation between these two estimates (\cite[Theorem (2.9)]{SW2}) gives
  $S_{\ast}^{\lambda}(H)$ is bounded on $L^2(\RN, \,  (1+|x|)^{-\alpha})$ for any $0<\alpha\leq 1$ as long as $\lambda>0$.
This proves  the sufficient part of Theorem~\ref{th1.2}.
\end{proof}

To complete  the proof of the sufficiency part of Theorem~\ref{th1.2},
it remains to prove Proposition~\ref{prop4.1}.

\newcommand{\fS}{\mathfrak S}

 \subsection{Weighted inequality  for the square function}
 In this subsection, we establish Proposition~\ref{prop4.1}.  For the purpose we decompose $\mathfrak S_\delta$ into high and low frequency parts.
 Let us set
 \begin{eqnarray*}
  \mathfrak S^{l}_{\delta}f(x) &=:& \Big( \int_{1/2}^{\delta^{-1/2}} \Big|\phi\Big(\delta^{-1}\Big(1-{ {H}\over t^2} \Big)\Big)f(x)\Big|^2
   {dt\over t}  \Big)^{1/2}, \\[2pt]
   \mathfrak S^{h}_{\delta}f(x) &=:& \Big( \int_{\delta^{-1/2}}^\infty \Big|\phi\Big(\delta^{-1}\Big(1-   { {H}\over t^2} \Big) \Big)f(x)\Big|^2
   {dt\over t}\Big)^{1/2}.
  \end{eqnarray*}
Since the first eigenvalue of the Hermite operator is larger than or equal to $1$,    $\phi\big(\delta^{-1}\big(1-{ {H}\over t^2} \big) \big)=0$
if $t\le 1$ because $\supp \phi\subset (2^{-3}, 2^{-1})$ and $\delta\le 1/2$. Thus, it is clear that
\begin{eqnarray}\label{e4.8}
\mathfrak S_\delta f(x)   &\le &
    \mathfrak S^{l}_{\delta}f(x)+\mathfrak S^{h}_{\delta}f(x).
  \end{eqnarray}

  \noindent
In order to prove   Proposition~\ref{prop4.1}  it is sufficient to
show  the following.

\begin{lemma} \label{le4.2}
Let $A_{\alpha, n}^{\epsilon}(\delta)$ be given by  \eqref{bbb}.
Then,  for all  $0<\delta\leq 1/2$ and $0<\epsilon \leq 1/2,$ we have the following estimates:
 \begin{eqnarray}
 \label{et12}
  \int_{\RR^n} |\mathfrak S_{\delta}^{l}f(x)|^2(1+|x|)^{-\alpha}dx &\leq& C\delta A_{\alpha, n}^{\epsilon}(\delta)
  \int_{\RR^n} |f(x)|^2(1+|x|)^{-\alpha}dx,
  \\
 \label{et121}
  \int_{\RR^n} |\mathfrak S_{\delta}^{h}f(x)|^2(1+|x|)^{-\alpha}dx &\leq& C\delta A_{\alpha, n}^{\epsilon}(\delta)
  \int_{\RR^n} |f(x)|^2(1+|x|)^{-\alpha}dx.
  \end{eqnarray}
\end{lemma}

Both of the proofs of the estimates \eqref{et12} and \eqref{et121}  rely on the generalized trace lemmata,
 Lemma~\ref{le3.2} and Lemma~\ref{le33.1}.
Though, there are distinct differences in their proofs.  As for \eqref{et12} we additionally use  the estimate \eqref{e1.8}
 which is efficient for the low frequency part.
Regarding  the estimate  \eqref{et121} we use the spatial localization argument which is based on the finite speed of
propagation of the Hermite wave operator $\cos(t\!\SH\,)$.
Similar strategy has been used to related  problems, see for example \cite{CLSY}. In this regards,  our proof of the
estimate \eqref{et121} is similar to that in \cite{CRV}.
In high frequency regime the localization strategy becomes more advantageous since the associated kernels enjoy tighter
localization. This allows us to handle the weight
$(1+|x|)^{-\alpha}$ in an easier way.   The choice of $\delta^{-\frac12}$  in the definitions of $\fS_\delta^l$, $\fS_\delta^h$
 is made by optimizing the estimates
which result from two different  approaches, see Remark \ref{re:aboutdelta}.

 \subsection{Proof of \eqref{et12}: low frequency part}
We start with  Littlewood-Paley decomposition associated with the operator $H$.
	Fix a  function $\varphi\in C^{\infty}  $  supported in $\{ 1\leq |s|\leq 3\}$ such that $\sum_{-\infty}^{\infty}\varphi(2^ks)=1$
	on ${\mathbb R}\backslash \{0\}$.
	By the spectral theory we have
	\begin{eqnarray}\label{e4.10}
	\sum_k \varphi_k(\!\sqrt{H}\,)f:=\sum_k\varphi(2^{-k}\sqrt{H}\,)=f\,.
	\end{eqnarray}
	for any 	$f\in L^2(\RN)$. Using  \eqref{e4.10}, we get
\Be  \label{e4.11}
 | \fS^{l}_{\delta}f(x)|^2
 \leq
  C\sum_{0\leq k\leq 1-{\rm log}_2{\sqrt{\delta} } }\int_{2^{k-1}}^{2^{k+2}}
  \Big|\phi\left(\delta^{-1}\left(1-{ {H}\over t^2} \right) \right)\varphi_k(\!\sqrt{H}\,)f(x)\Big|^2
  {dt\over t}
\Ee
   for $f\in L^2(\RN)\cap L^2(\RN, (1+|x|)^{-\alpha})$.
  To exploit disjointness of the spectral support $\phi\big(\delta^{-1}\big(1-{ {H}\over t^2} \big) \big)$ we make additional decomposition in $t$.
  For  $k\in{\mathbb Z}$ and  $i=0, 1, \cdots, i_0=[{8/\delta}] +1$
we set
\begin{eqnarray}
\label{e-Ii}
I_i=\left[2^{k-1} + i 2^{k-1}\delta, \, 2^{k-1} + (i+1)2^{k-1}\delta\right],
\end{eqnarray}
   so that
  $[2^{k-1}, 2^{k+2}]\subseteq \cup_{i=0}^{i_0} I_i$. Define  $\eta_i$  adapted to the interval $I_i$ by setting
  \begin{eqnarray}\label{e-eta}
\eta_i(s) = \eta\left( i+{ 2^{k-1} -s\over 2^{k-1}\delta}\right),
  \end{eqnarray}
where $\eta\in C_c^{\infty}(-1, 1)$ and $\sum_{i\in {\mathbb Z}} \eta(\cdot-i)=1$.
For simplicity we also set
\[  \phi_{\delta}(s\,):=\phi(\delta^{-1}\big(1-s^2  ) ).\]
We observe that,  for  $t\in I_i$,
$
\phi_{\delta}\left({s/t}\right)\eta_{i'} (s)\not=0$ only if $i-i\delta-3\leq i'\leq i+i\delta+3.
$
Hence,  for $t\in I_i$ we  have
$$  \phi_{\delta}\big( t^{-1}  \sqrt H \big)
  \varphi_k(\!\sqrt{H}\,)
 =  \sum_{i'=i-10}^{i+10}
 \phi_{\delta}\big({t^{-1}  \sqrt H}\big)
  \varphi_k(\!\sqrt{H}\,)
 \eta_{i'}  (\!\sqrt{H}\,),
$$
  and thus
   \begin{eqnarray*}
  & &\hspace{-0.8cm}
 \int_{2^{k-1}}^{2^{k+2}}
   \Big|\phi_{\delta}\big({t^{-1}  \sqrt H}\big)
  \varphi_k(\!\sqrt{H}\,)f\Big|^2    {dt\over t} \\
  &\leq& C\sum_i\sum_{i'=i-10}^{i+10}\int_{I_i} \Big|
 \phi_{\delta}\big({t^{-1}   \sqrt H}\big)
  \varphi_k(\!\sqrt{H}\,)
 \eta_{i'}  (\!\sqrt{H}\,)f\Big|^2   {dt\over t}.
 \end{eqnarray*}
  Substituting this into \eqref{e4.11}, we have that
  \[
   |\mathfrak S^{l}_{\delta}f(x)|^2 \leq C \!\!\! \sum_{0\leq k\leq 1-{\rm log}_2{\sqrt{\delta}}}\sum_i\sum_{i'=i-10}^{i+10}\int_{I_i}
  \left|\phi_{\delta}\big({t^{-1}   \sqrt H}\big)
  \eta_{i'}  (\!\sqrt{H}\,)\varphi_k(\!\sqrt{H}\,)f(x)\right|^2    {dt\over t}.
  \]
   Now we claim  that,   for $1\leq t\leq  \delta^{-1/2}$,
 \begin{equation}
 \label{e4.16}
 \int_{\RR^n} |\phi_{\delta}\big({t^{-1} \sqrt H}\big)g(x)|^2 (1+|x|)^{-\alpha} dx
  \leq C A_{\alpha, n}^{\epsilon}(\delta) \int_{\RR^n} |(1+H)^{-\alpha/4}g(x)|^2 dx.
\end{equation}
Before we begin to prove it, we show that this concludes the proof of estimate~\eqref{et12}.
Combining \eqref{e4.16} with the preceding inequality,  we see that
$  \int_{\RR^n} |\mathfrak S_{\delta}^{l}f(x)|^2(1+|x|)^{-\alpha}dx$ is bounded by
\[
CA_{\alpha, n}^{\epsilon}(\delta) \sum_{0\leq k\leq 1-{\rm log}_2{\sqrt{\delta}}} \sum_i\sum_{i'=i-10}^{i+10}\int_{I_i} \int_{\RR^n}
  \left|  \eta_{i'}  (\!\sqrt{H}\,)\varphi_k(\!\sqrt{H}\,)(1+H)^{-\alpha/4}f(x)\right|^2  dx  {dt\over t}.
  \]
Since the length of interval $I_i$ is comparable to $2^{k-1}\delta$, taking integration in $t$ and using disjointness of the spectral supports,
we get
    \begin{eqnarray*}
   \int_{\RR^n} |\mathfrak S_{\delta}^{l}f(x)|^2(1+|x|)^{-\alpha}dx
  &\leq& C\delta  A_{\alpha, n}^{\epsilon}(\delta) \int_{\RR^n}
  \left|  (1+H)^{-\alpha/4}f(x)\right|^2  dx.  
   \end{eqnarray*}
This, being combined with  Lemma~\ref{prop2.1}, yields  the desired estimate \eqref{et12}.

We now show the estimate
 \eqref{e4.16}. Let us consider the equivalent estimate
  \begin{equation}
 \label{e4.161}
 \int_{\RR^n} |\phi_{\delta}\big({t^{-1}   \sqrt H}\big)(1+H)^{\alpha/4}g(x)|^2 (1+|x|)^{-\alpha} dx
  \leq C A_{\alpha, n}^{\epsilon}(\delta) \int_{\RR^n} |g(x)|^2 dx.
\end{equation}
We first show the estimate for the case  $n\geq 2$.
  Let $N=8[t]+1$. Note that $\supp \phi_{ \delta}\big({\cdot/t}\big)\subset [N/4,N]$.
By  Lemma~\ref{le3.2},
 \begin{eqnarray*}
&&\hspace{-1cm}
  \int_{\RR^n} |\phi_{\delta}\big({t^{-1}   \sqrt H}\big)(1+H)^{\alpha/4}g(x)|^2 (1+|x|)^{-\alpha} dx
 \\
 &\leq&  CN
 \Big\|\phi_{\delta}\big({ t^{-1}N u}\big)(1+N^2 u^2)^{\alpha/4}\Big\|^2_{N^2,2}\int_{\RR^n} |g(x)|^2 dx.
  \end{eqnarray*}
 We now estimate $\|\phi_{\delta}\big({ t^{-1}N u}\big)(1+N^2 u^2)^{\alpha/4}\|^2_{N^2,2}$.
 Note that
$
\supp \,\phi_{\delta}\big({ t^{-1}N u}\big)\subset [\frac{t\sqrt{1-\delta}}{N},\, \frac{t\sqrt{1-\delta/4}}{N} ].
$
Since  the length of the interval $[\frac{t\sqrt{1-\delta}}{N},\, \frac{t\sqrt{1-\delta/4}}{N} ]\sim \delta$ and  $N\sim t\leq \delta^{-1/2}$,  we get
\begin{eqnarray}\label{ccc}  \hspace{0.5cm}
 &&\hspace{-1cm}\big\|\phi_{\delta}\big({ t^{-1}N u}\big)(1+N^2 u^2)^{\alpha/4}\big\|^2_{N^2,2} \nonumber\\
&\leq&
\nonumber
  \big\|\phi_{\delta}\big({ t^{-1}N u}\big)(1+N^2 u^2)^{\alpha/4}
\big\|^2_\infty \big\|\chi_{[\frac{t\sqrt{1-\delta}}{N},\, \frac{t\sqrt{1-\delta/4}}{N}]}\big\|^2_{N^2,2}\nonumber\\
&\leq&  C N^{\alpha-2}.
 \end{eqnarray}
Thus, noting $1/2\leq t\leq  \delta^{-1/2}$ and $\alpha>1$,  we obtain
\begin{eqnarray*}
& &\hspace{-0.8cm}
\int_{\RR^n} |\phi_{\delta}\big({t^{-1}   \sqrt H}\big)(1+H)^{\alpha/4}g(x)|^2 (1+|x|)^{-\alpha} dx
\\
 & \leq & C N^{\alpha-1}\int_{\RR^n} |g(x)|^2 dx \leq C\delta^{1/2-\alpha/2}\int_{\RN} |g(x)|^2 dx,
  \end{eqnarray*}
which  gives \eqref{e4.16} in the   dimensional case $n\geq 2$.

 Next we  prove  \eqref{e4.16} with $n=1$. Let $0\leq\alpha<1$ and  $N=8[t]+1$. Note that $\supp \phi_{ \delta}\big({\cdot/t}\big)\subset [N/4,N]$.
By  Lemma~\ref{le33.1},  for any $\varepsilon>0$ we have
\begin{eqnarray*}
{\rm LHS\ of\ } \eqref{e4.16}    \leq C_\varepsilon N^{\frac{\alpha}{1+\varepsilon}}
\Big\|\phi_{\delta}\big({ t^{-1}N u}\big)(1+N^2 u^2)^{\alpha/4}\Big\|^2_{N^2,\frac{2(1+\varepsilon)}{\alpha}}\int_{\mathbb R} |g(x)|^2 dx.
  \end{eqnarray*}
As before, in the same manner as  in \eqref{ccc} we have
\[
 \|\phi_{\delta}\big({ t^{-1}N u}\big)(1+N^2 u^2)^{\alpha/4}\|^2_{N^2,\frac{2(1+\varepsilon)}{\alpha}}
\leq C  N^{\frac{\alpha(\varepsilon-1)}{1+\varepsilon}},
 \]
so it follows that
\begin{eqnarray*}
\int_{\RR^n} |\phi_{\delta}\big({t^{-1}  \sqrt H}\big)(1+H)^{\alpha/4}g(x)|^2 (1+|x|)^{-\alpha} dx
   &\leq&C\delta^{-\frac{\alpha\varepsilon}{2(1+\varepsilon)}}\int_{\RR^n} |g(x)|^2 dx
  \end{eqnarray*}
  because  $1\leq t\leq \delta^{-1/2}$.
 This gives \eqref{e4.16} with $n=1$ and the  proof of estimate~\eqref{et12} is completed.
 \hfill{} $\Box$

 \subsection{Proof of  \eqref{et121}: high frequency part}
 We now make use of
 the finite speed of propagation of  the  wave operator $\cos(t\SH\,\,)$. From \eqref{e2.6}, it is known
 (see for example \cite{CouS})
 that
 the kernel of the operator $\cos(t\SH\,\,)$ satisfies
 \begin{equation}\label{e4.21}
\supp  K_{\cos(t\SH\,\,)} \subseteq \D(t):=\{ (x,\, y)\in \RN\times \RN: |x-y| \le t \}, \quad \forall t> 0\,.
\end{equation}
For any even function $F$ with $ \widehat{F}\in L^1({\mathbb R})$ we have
\[
F(t^{-1} \SH\,\,) =\frac{1}{2\pi}\int_{-\infty}^{+\infty}
  \widehat{F}(\tau) \cos(\tau\, t^{-1} \SH\,\,) \;d\tau.
  \]
 Thus from the above  we have
 \begin{equation}
 \label{e4.210}
 \supp K_{F(t^{-1}\!\SH\,\,)} \subseteq \D(t^{-1}r)
 \end{equation}
 whenever $\supp\widehat{F} \subseteq [-r,r]$.
This will be used in what follows.

Fixing an even function  $\vartheta\in C_c^{\infty}$ which is identically one on
$\{|s| \leq 1 \}$  and supported on $\{|s| \leq 2 \}$,   let us set  $j_0=[-\log_2\delta]-1$ and
$
\zeta_{j_0}(s):=\vartheta(2^{-j_0} s)$ and   $ \zeta_j(s):=\vartheta(2^{-j} s)-\vartheta(2^{-j+1} s)$ for $ j> j_0$. Then, we clearly have
\begin{eqnarray*}
\label{e4.23}
1\equiv \sum_{j\geq j_0 }\zeta_j(s), \ \ \ \ \forall s>0.
\end{eqnarray*}
Recalling that $\phi_{\delta}(s)={ \phi(\delta^{-1}(1-s^2) )}$,  for
 $ j\geq j_0$ we set
   \begin{eqnarray}\label{e4.24}
  \phi_{\delta,j}(s)={1\over 2\pi}  \int_{-\infty}^{\infty}\zeta_j(u)
  {\widehat{ \phi_{\delta} }}(u) \cos(s u) du.
  \end{eqnarray}
 By a routine computation  it can be verified that
    \begin{eqnarray}\label{e4.25}
  | \phi_{\delta,j}(s)|\leq
  \left\{
  \begin{array}{ll}
  \qquad C_N 2^{(j_0-j)N} ,  & |s|\in [1-2\delta,1+2\delta],\\[8pt]
   C_N  2^{j-j_0}  (1+ 2^j |s-1|)^{-N}, & \qquad {\rm otherwise},
  \end{array}
  \right.
  \end{eqnarray}
  for any $N$ and all  $j\geq j_0$   (see also   \cite[page 18]{C1}).
 By the Fourier inversion formula, we have
 \begin{eqnarray}\label{e4.26}
 \phi\left(\delta^{-1}\left(1-{s}^2 \right)\right)= \sum_{j\geq j_0}\phi_{\delta,j}(s), \ \ \ \ s>0.
  \end{eqnarray}
By the finite speed propagation property \eqref{e4.210}, we particularly have
\Be \label{e4.27}
  \supp K_{\phi_{\delta,j}(\!\sqrt{H}/t)}\subseteq \D(t^{-1}2^{j+1}) =\left\{(x,y)\in \RN\times \RN:  \ |x-y|\leq 2^{j+1}/t\right\}.
\Ee

Now from \eqref{e4.10}, it follows that
	 \Be \label{chen:e1}
 | \mathfrak S^{h}_{\delta}f(x)|^2
 \leq
  5\sum_{k\geq 1-{\rm log}_2 \sqrt{\delta}}  \int_{2^{k-1}}^{2^{k+2}}
  \Big|\phi\left(\delta^{-1}\big(1-{t^{-2}{H}}\big) \right)\varphi_k(\!\sqrt{H}\,)f(x)\Big|^2
  {dt\over t}.
 \Ee
  For $k\geq 1-{\rm log}_2 \sqrt{\delta}$ and $j\geq j_0$, let us set
 \[ E^{k, j}(t):=
  \Big\langle  \Big|\phi_{\delta,j}\big({t^{-1} \!\! \sqrt H}\,\big)
  \varphi_k(\!\sqrt{H}\,)f(x)\Big|^2, \  (1+|x|)^{-\alpha}\Big\rangle .\]
Using  the above inequality~\eqref{chen:e1},  \eqref{e4.26}, and Minkowski's inequality,  we have
 \Be \label{e4.28}
  \int_{\RN} | \mathfrak S^{h}_{\delta}f(x)|^2 (1+|x|)^{-\alpha} dx
  \le
  C\!\!\sum_{k\geq 1-{\rm log}_2 \sqrt{\delta}} \Big( \sum_{j\geq j_0}   \Big( \int_{2^{k-1}}^{2^{k+2}}
 E^{k, j}(t){dt\over t} \Big)^{1/2}\,\Big)^{2}.
 \Ee

In order to make use of the localization property \eqref{e4.27} of the kernel, we need to decompose $\mathbb R^n$
into disjoint cubes of side length  $2^{j-k+2}$.
For a given $k\in{\mathbb Z}, j\geq j_0$, and $\mathbf m=(\mathbf m_1, \cdots, \mathbf m_n)\in \mathbb Z^n,$ let us set
 $$
Q_{\mathbf m}= \Big[ 2^{j-k+2}\big(\mathbf m_1 -{1\over 2}\big), \ 2^{j-k+2}\big(\mathbf m_1 +{1\over 2}\big)\!\!\Big)  \times
 \cdots \times \Big[ 2^{j-k+2}\big(\mathbf m_n -{1\over 2}\big), \ 2^{j-k+2}\big(\mathbf m_n +{1\over 2}\big)\!\!\Big),
 $$
which are disjoint dyadic cubes centered at  $2^{j-k+2}\mm$ with side length  $2^{j-k+2}$. Clearly,
 $\RN=\cup_{\mathbf m\in\mathbb Z^n}Q_{\mathbf m}$.  For each $\mathbf m$, we define  $\widetilde{Q_{\mathbf m}}$ by setting
 $$\widetilde{Q_{\mathbf m}}:=
\bigcup_{{\mm'}\in \mathbb Z^n:\ {\rm dist}\, (Q_{\mm'}, Q_{\mathbf m})\leq  \sqrt n 2^{j-k+3}  } Q_{\mm'},
$$
and denote
$$
\mathbf M_0:= \big\{\mathbf m\in \mathbb Z^n: Q_0\cap \widetilde{Q_{\mathbf m}} \not= \emptyset \big\}.
$$
For $t\in[ 2^{k-1}, 2^{k+2}]$
it follows by \eqref{e4.27}   that  $ \bchi_{Q_{\mathbf m}}  \phi_{\delta,j}\big({t^{-1}   \sqrt H}\big)
 \bchi_{Q_{\mm'}}=0$ if  $\widetilde{Q_{\mathbf m}}\cap Q_{\mm'}=\emptyset$   for every $j,k$. Hence, it is clear that
 \begin{eqnarray*}
  \phi_{\delta,j}\big({t^{-1}   \sqrt H}\big)
   \varphi_k(\!\sqrt{H}\,)f
    =
  \!\!\! \sum_{\mm, \mm': \,  {\rm dist}\, (Q_{\mm}, Q_{\mm'}) <t^{-1}2^{j+2}}    \!\!\!   \!\!\!   \bchi_{Q_{\mathbf m}}
   \phi_{\delta,j}\big({t^{-1}   \sqrt H}\big) \bchi_{Q_{\mm'}}
   \varphi_k(\!\sqrt{H}\,)f,
    \end{eqnarray*}
 which gives
 \Be \label{e4.30}
 E^{k, j}(t)  \leq    C \sum_m
   \Big\langle
  \big|  \bchi_{ Q_{\mathbf m}}
   \phi_{\delta,j}\big({t^{-1}  \sqrt H}\big)   \bchi_{\widetilde{{Q}_\mm}} \varphi_k(\!\sqrt{H}\,)  f(x)
 \big|^2, \,    (1+|x|)^{-\alpha}  \Big\rangle.
  \Ee
 To exploit orthogonality generated by the disjointness of spectral support, we further decompose $\phi_{\delta,j}$
  which  is not compactly supported. We choose an even
  function $\theta\in C^\infty_c(-4, 4) $
 such that $\theta(s)=1$  for $s\in (-2, 2)$. Set
 \begin{equation}\label{e4.31}
 { \psi}_{0, \delta}(s):=\theta(\delta^{-1}(1-s)),  \qquad
 { \psi}_{\ell, \delta}(s):=\theta(2^{-\ell}\delta^{-1}(1-s)) - \theta(2^{-\ell+1}\delta^{-1}(1-s))
\end{equation}
 for all $\ell\geq 1$ such that $1=\sum_{\ell=0}^{\infty} { \psi}_{\ell, \delta}(s)$ and
 $
  \phi_{\delta,j}(s)=  \sum_{\ell=0}^{\infty}  \big ( { \psi}_{\ell, \delta}\phi_{\delta,j}\big)(s)
  $ for all $s>0.
 $
We put it into  \eqref{e4.30} to write
   \begin{eqnarray}\label{e4.32}\hspace{0.5cm}
 \Big( \int_{2^{k-1}}^{2^{k+2}}
 E^{k, j}(t){dt\over t}\Big)^\frac12
 &\leq & \sum_{\ell=0}^{\infty}
  \Big(  \sum_{\mathbf m}  \int_{2^{k-1}}^{2^{k+2}}  E^{k, j,\,\ell}_\mm(t)
   {dt\over t}\Big)^{1/2},
  \end{eqnarray}
  where
  \[  E^{k, j,\,\ell}_\mm(t) =  \Big\langle
  \big|  \bchi_{ Q_{\mathbf m}}
 \big ( { \psi}_{\ell, \delta}\phi_{\delta,j}\big)\big({t^{-1} \!\! \sqrt H}\big)
 \bchi_{\widetilde {Q_{\mathbf m}}} \varphi_k(\!\sqrt{H}\,)  f(x)\big|^2, \,   (1+ |x|)^{-\alpha}  \Big\rangle.\]

Recalling \eqref{e-Ii} and \eqref{e-eta},  we observe that,  for every $t\in I_i$,
it is possible that  $
{\psi}_{\ell, \delta} \left({s/t}\right)\eta_{i'} (s)\not=0$ only when  $i-2^{\ell+6}\leq i'\leq i+2^{\ell+6}.$
Hence,
$$ ( { \psi}_{\ell, \delta}\phi_{\delta,j})({t^{-1}   \sqrt H})
 =   \sum_{i'=i-2^{\ell+6}}^{i+2^{\ell+6}}
 ({ \psi}_{\ell, \delta}\phi_{\delta,j})({t^{-1}   \sqrt H})
 \eta_{i'}  (\!\sqrt{H}\,), \ \ \ \  t\in I_i.
 $$
  From this and Cauchy-Schwarz's inequality  we have
 \[
  E^{k, j,\,\ell}_\mm(t) \le C 2^\ell    \sum_{i'=i-2^{\ell+6}}^{i+2^{\ell+6}}  E^{k, j,\,\ell}_{\mm, i'}(t)
   \]
 for $t\in I_i$ where
 \[  E^{k, j,\,\ell}_{\mm, i'}(t):=\Big\langle
  \big|  \bchi_{ Q_{\mathbf m}}
 \big( { \psi}_{\ell, \delta}\phi_{\delta,j}\big)\big({t^{-1}   \sqrt H}\big) \eta_{i'}  (\!\sqrt{H}\,)
\big[\bchi_{\widetilde {Q_{\mathbf m}}} \varphi_k(\!\sqrt{H}\,)  f\big](x)\big|^2, \,    (1+|x|)^{-\alpha}  \Big\rangle.
\]
 Combining this with \eqref{e4.32}, we get
\Be\label{chen:e2}
\Big( \int_{2^{k-1}}^{2^{k+2}}
 E^{k, j}(t){dt\over t}\Big)^\frac12\le C  \sum_{\ell=0}^{ \infty} 2^{\ell/2}\Big( \sum_{\mm }    \sum_{i}  \int_{I_i}
  \sum_{i'=i-2^{\ell+6}}^{i+2^{\ell+6}}
  E^{k, j,\,\ell}_{\mm, i'}(t){dt\over t}\Big)^{1/2}.
  \Ee
  To continue, we
distinguish  two cases:  $j> k$;  and $j\leq k$.  In the latter case the associated cubes have side length $\le 4$ so that
the weight $(1+|x|)^\alpha$  behaves like a constant on each cube $Q_\mm$, so the desired estimate is easier to obtain.
The first case is more involved and we need to distinguish several cases which we separately handle.

\newcommand{\cI}{\mathcal I}

\subsubsection{\bf Case  $j> k$.} From the above inequality \eqref{chen:e2}   we now have
  \begin{eqnarray}\label{mmmmm}
\Big( \int_{2^{k-1}}^{2^{k+2}}
 E^{k, j}(t){dt\over t}\Big)^\frac12\le
  I_1(j,k) + I_2(j,k) + I_3(j,k),
  \end{eqnarray}
  where
 \begin{align}
  I_1(j,k)&:= \sum_{\ell=0}^{  [-{\rm log_2{\delta}}]-3} 2^{\ell/2}\Big(   \sum_{\mm\in \mathbf M_0 }    \sum_{i}  \int_{I_i}
  \sum_{i'=i-2^{\ell+6}}^{i+2^{\ell+6}}
  E^{k, j,\,\ell}_{\mm, i'}(t){dt\over t}\Big)^{1/2}, \label{e.I1}
\\
  I_2(j,k) &:=   \sum_{\ell=  [-{\rm log_2{\delta}}]-2}^{\infty} 2^{\ell/2} \Big( \sum_{\mm\in \mathbf M_0 }    \sum_{i}  \int_{I_i}
  \sum_{i'=i-2^{\ell+6}}^{i+2^{\ell+6}}
  E^{k, j,\,\ell}_{\mm, i'}(t){dt\over t}\Big)^{1/2}, \label{e.I2}
  \\
     I_3(j,k) &:=  \sum_{\ell=  0}^{\infty} 2^{\ell/2}\Big( \sum_{\mm\not\in \mathbf M_0 }   \sum_{i}  \int_{I_i}
  \sum_{i'=i-2^{\ell+6}}^{i+2^{\ell+6}}
  E^{k, j,\,\ell}_{\mm, i'}(t){dt\over t}\Big)^{1/2}. \label{e.I3}
  \end{align}

We first consider the estimate for $I_1(j, k)$ which is the major one.
The estimates for $I_2(j, k)$, $I_3(j, k)$ are to be obtained similarly but easier.
In fact, concerning $I_3(j, k)$, the weight $(1+|x|)^{-\alpha}$ behave as if it were a constant, and
the bound on $I_2(j, k)$ is much smaller than what we need to show because of rapid decay of the associated multipliers.

\noindent
{\underline {\it Estimate of the term $I_1(j, k)$ }}.
We claim that, for any $N>0$,
  \Be \label{ecccc}
   I_1(j,k)
   \le C_N 2^{(j_0-j)N} \left(\delta A_{\alpha, n}^{\epsilon}(\delta)\right)^{1/2} \Big(\int_{\RN}
| \varphi_k(\!\sqrt{H}\,)  f (x)|^2(1+|x|)^{-\alpha}dx \Big)^{1/2}
 \Ee
where $A_{\alpha, n}^{\epsilon}(\delta)$ is defined in~\eqref{bbb}.

Let us first  consider the  case $n\geq 2$. For \eqref{ecccc}, it suffices  to show
\begin{eqnarray}\label{e4.36}\hspace{1cm}
 E^{k, j,\,\ell}_{\mm, i'}(t)
&\leq &
  C_N 2^{-  \ell N }   2^{(j_0-j)N}  2^{ k}  \delta
   \int_{\RN}\big| \eta_{i'}  (\!\sqrt{H}\,)
\big[\bchi_{\widetilde {Q_{\mathbf m}}} \varphi_k(\!\sqrt{H}\,)  f\big](x)\big|^2dx
\end{eqnarray}
for any $N>0$ while $t\in I_i$ being fixed and $i-2^{\ell+6}\leq i'\leq i+2^{\ell+6}.$
Indeed, since the supports of $\eta_i$ are boundedly overlapping, \eqref{e4.36} gives
\begin{equation}
\label{e4.37}  \sum_{i}  \int_{I_i}
  \sum_{i'=i-2^{\ell+6}}^{i+2^{\ell+6}}
  E^{k, j,\,\ell}_{\mm, i'}(t){dt\over t}
 \lesssim
 C_N 2^{-  \ell (N-1) }   2^{(j_0-j)N}  2^{ k}  \delta^2\big \|\bchi_{\widetilde {Q_{\mathbf m}}} \varphi_k(\!\sqrt{H}\,)  f  \big\|^2_{2} .
 \end{equation}
Recalling \eqref{e.I1}, we take summation over $\ell$ and $\mm\in \mathbf M_0$ to get
\[ I_1(j,k)
\le   C_N 2^{j_0\alpha/2} 2^{(j_0-j)(N-\alpha)/2}  2^{k(1-\alpha)/2}
\delta \,
 \Big(2^{(k-j)\alpha}\sum_{\mm\in \mathbf M_0} \big\|
 \bchi_{\widetilde {Q_{\mathbf m}}} \varphi_k(\!\sqrt{H}\,)  f  \big\|^2_{2} \Big)^{1/2}.
  \]
Since $j>k$ and $\mm\in \mathbf M_0$, we note that $(1+|x|)^{\alpha}\leq C2^{(j-k)\alpha}$  if $x\in Q_{\mathbf m}$.
It follows that
\begin{equation}
\label{e.compare} 2^{(k-j)\alpha}\sum_{\mm\in \mathbf M_0} \big\|
 \bchi_{\widetilde {Q_{\mathbf m}}} \varphi_k(\!\sqrt{H}\,)  f  \big\|^2_{2}
 \le C \int_{\RN} | \varphi_k(\!\sqrt{H}\,)  f (x)|^2(1+|x|)^{-\alpha}dx.
\end{equation}
Noting that  $j_0=[-\log_2 \delta]-1$
and $k\geq [- {1\over 2}\log_2 \delta  ]$, we obtain
 \begin{eqnarray*}
I_1(j,k)
 &\leq& C_N 2^{(j_0-j)(N-\alpha)/2} \delta^{3/4-\alpha/4}\left(\int_{\RN}
| \varphi_k(\!\sqrt{H}\,)  f (x)|^2(1+|x|)^{-\alpha}dx \right)^{1/2},
  \end{eqnarray*}
which clearly gives \eqref{ecccc} since $N>0$ is arbitrary.

We now proceed to prove \eqref{e4.36}.
Note  that  $ \supp { \psi}_{\ell, \delta} \subseteq (1-2^{\ell+2}\delta, 1+2^{\ell+2}\delta)$,
and so $\supp \left (\psi_{\ell, \delta}\phi_{\delta,j}\right)\left({\cdot/t}\right)\subset [t(1-2^{\ell+2}\delta), t(1+2^{\ell+2}\delta)]$.
Thus, setting $R=[t(1+2^{\ell+2}\delta)]$,
by  Lemma~\ref{le3.2} we get
\Be
\label{e.4Ejk}
E^{k, j,\,\ell}_{\mm, i'}(t)
\le
R \|({  \psi}_{\ell, \delta}\phi_{\delta,j})\big(R \cdot/t\big)\|^2_{R^2,2} \int_{\RN}\left| \eta_{i'}  (\!\sqrt{H}\,)
\big[\bchi_{\widetilde {Q_{\mathbf m}}} \varphi_k(\!\sqrt{H}\,)  f\big](x)\right|^2dx
\Ee
for $0\leq \ell\leq [-{\rm log_2{\delta}}]-3$.
We note that the support of
$
({  \psi}_{\ell, \delta}\phi_{\delta,j})\big(R \cdot/t\big)$ is contained in $\subset [R^{-1}{t(1-2^{\ell+2}\delta)},\,  R^{-1}{t(1+2^{\ell+2}\delta)}]
$
and  $R^2\delta\geq 1$. Thus, we get
\[
 \|({  \psi}_{\ell, \delta}\phi_{\delta,j})\big((1+2^{\ell+2}\delta) \cdot\big)\|_{R^2,2}
\leq C \|({  \psi}_{\ell, \delta}\phi_{\delta,j})\big((1+2^{\ell+2}\delta) \cdot\big)\|_\infty  \left({2^{\ell+3}\delta}\right)^{1/2}.
 \]
 On the other hand, if $\ell\geq 1$,   then  ${ \psi}_{\ell, \delta}(s)=0$  for $s\in (1-2^{\ell}\delta, 1+2^{\ell}\delta)$,
which together with~\eqref{e4.25} and   $j_0=[-\log_2\delta]-1$ shows that
  \begin{eqnarray}
  \label{chen:e3}
  \|{  \psi}_{\ell, \delta}\phi_{\delta,j}\big((1+2^{\ell+2}\delta) \cdot\big)\|_{L^\infty} \le  C_N
   2^{(j_0-j)N} 2^{- \ell N }, \quad  \ell\geq 0
\end{eqnarray}
 for any $N<\infty$ .  Since $R\sim 2^k$, combining these two estimates with \eqref{e.4Ejk}
we get  the desired~\eqref{e4.36}.

Now we prove \eqref{ecccc}  with $n=1$.  This case can be handled in the same manner as before, so we shall be brief.
The only difference is that we  use Lemma~\ref{le33.1} instead of Lemma~\ref{le3.2}. Indeed, by following the same
 argument in the above and using  Lemma~\ref{le33.1}, we get
\begin{eqnarray*}
 E^{k, j,\,\ell}_{\mm, i'}(t) &\leq &
  C_N 2^{-  \ell N }   2^{(j_0-j)N} \big(\delta 2^{\ell + k} \big)^{\alpha\over 1+\epsilon}
   \int_{\RN}\big| \eta_{i'}  (\!\sqrt{H}\,)
\big[\bchi_{\widetilde {Q_{\mathbf m}}} \varphi_k(\!\sqrt{H}\,)  f\big](x)\big|^2dx
\end{eqnarray*}
for any $N>0$. Once we have the above estimate, one can deduce  the estimate
~\eqref{ecccc} without difficulty.

\noindent
{\underline {\it Estimate of the term $I_2(j, k)$}}.
As is clear in the decomposition of $E^{k, j},$ the term $I_2(j, k)$ is a tail part and we can obtain an estimate
 which is stronger than we need to have.  In fact,
we  show
\begin{eqnarray}
\label{e.est-I2}
   I_2(j,k)
   &\leq& C_N 2^{(j_0-j)N} \delta^N \Big(\int_{\RN}
| \varphi_k(\!\sqrt{H}\,)  f (x)|^2(1+|x|)^{-\alpha}dx \Big)^{1/2}
  \end{eqnarray}
  for any $N>0$.
 Indeed, we clearly have
  \begin{eqnarray*}
E^{k, j,\,\ell}_{\mm, i'}(t)&\leq & \left\|
 \left ( { \psi}_{\ell, \delta}\phi_{\delta,j}\right)\left({t^{-1} \!\! \sqrt H}\right) \eta_{i'}  (\!\sqrt{H}\,)
\big[\bchi_{\widetilde {Q_{\mathbf m}}} \varphi_k(\!\sqrt{H}\,)  f\big]\right\|^2_2\\
&\leq& \|({  \psi}_{\ell, \delta}\phi_{\delta,j})\big( \cdot/t\big)\|^2_{\infty} \int_{\RN}\left| \eta_{i'}  (\!\sqrt{H}\,)
\big[\bchi_{\widetilde {Q_{\mathbf m}}} \varphi_k(\!\sqrt{H}\,)  f\big](x)\right|^2dx.
\end{eqnarray*}
From the definition of ${ \psi}_{\ell, \delta}$ and  \eqref{e4.25} we have
\[
\|({  \psi}_{\ell, \delta}\phi_{\delta,j})\big( \cdot/t\big)\|_{ \infty}\leq C_N  2^{j-j_0}(2^{j+\ell}\delta)^{-N}, \quad \ell\geq [-{\rm log_2{\delta}}]-2.
\]
Thus, it follows that
  \begin{eqnarray*}
E^{k, j,\,\ell}_{\mm, i'}(t)&\leq &  C_N 2^{2(j-j_0)}(2^{j+\ell}\delta)^{-2N}\int_{\RN}\left| \eta_{i'}  (\!\sqrt{H}\,)
\big[\bchi_{\widetilde {Q_{\mathbf m}}} \varphi_k(\!\sqrt{H}\,)  f\big](x)\right|^2dx.
\end{eqnarray*}
After putting  this in \eqref{e.I2} we take summation over  $\mm\in \mathbf M_0$ to obtain
 \begin{eqnarray*}
E_2(j,k)&\leq& C_N\delta^{1/2-N} 2^{(j-j_0)}2^{-Nj} 2^{(j-k)\alpha/2}
\\
  & & \qquad \times  \sum_{\ell=[-{\rm log_2{\delta}}]-2}^{\infty}
 2^{- \ell (N-2) } \Big(2^{(k-j)\alpha}  \sum_{\mm\in \mathbf M_0} \big\|
 \bchi_{\widetilde {Q_{\mathbf m}}} \varphi_k(\!\sqrt{H}\,)  f  \big\|^2_{2} \Big)^{1/2}.
  \end{eqnarray*}
As before we may use \eqref{e.compare} since  $j>k$. Since $j_0=[-\log_2 \delta]-1$
and $k\geq [- {1\over 2}\log_2 \delta  ]$, taking  sum over  $\ell$  we obtain \eqref{e.est-I2}.

\noindent
{\underline {\it Estimate of the term $I_3(j, k)$ }}. We now prove the estimate
\begin{eqnarray}
\label{e.est-I3}
   I_3(j,k)
    &\leq& C_N 2^{(j_0-j)N}   \delta^{1/2} \left(\int_{\RN}
| \varphi_k(\!\sqrt{H}\,)  f (x)|^2(1+|x|)^{-\alpha}dx \right)^{1/2}.
  \end{eqnarray}
We begin with making an observation  that
\begin{eqnarray}\label{e4.299}
 C^{-1} (1+|2^{j-k+2}\mm|) \leq 1+|x| \leq C (1+|2^{j-k+2}\mm|), \quad x\in Q_\mm
\end{eqnarray}
provided that $\mm\not\in \mathbf M_0$. Thanks to this observation  the estimates for  $E^{k, j,\,\ell}_{\mm, i'}(t)$
  is much simpler.  By \eqref{e4.299} it is clear that
	\begin{eqnarray*}
E^{k, j,\,\ell}_{\mm, i'}(t)  \leq  (1+|x_m|)^{-\al}  \left\|
 \left ( { \psi}_{\ell, \delta}\phi_{\delta,j}\right)\left({t^{-1} \!\! \sqrt H}\right)   \right\|_{2\to 2}^2
 \left\| \eta_{i'}  (\!\sqrt{H}\,)
\big[\bchi_{\widetilde {Q_{\mathbf m}}} \varphi_k(\!\sqrt{H}\,)  f\big]\right\|_2^2.
  \end{eqnarray*}
  Since  $\|
 ( { \psi}_{\ell, \delta}\phi_{\delta,j})({t^{-1} \!\! \sqrt H})   \|_{2\to 2}\le \| ( { \psi}_{\ell, \delta}\phi_{\delta,j})\|_\infty $,
  it follows from~\eqref{chen:e3} that we  have
  \[   E^{k, j,\,\ell}_{\mm, i'}(t)  \leq  C_N (1+|x_m|)^{-\al}  2^{2(j_0-j)N}2^{-2 \ell N}   \big\|
 \eta_{i'}  (\!\sqrt{H}\,)\big[\bchi_{\widetilde {Q_{\mathbf m}}} \varphi_k(\!\sqrt{H}\,)  f\big] \big\|_{2}^2.  \]
 Using this and disjointness of the spectral supports,  successively,  we get
\begin{eqnarray*}
& & \sum_{\mm\not\in \mathbf M_0 } \sum_{i}  \int_{I_i}
  \sum_{i'=i-2^{\ell+6}}^{i+2^{\ell+6}}
  E^{k, j,\,\ell}_{\mm, i'}(t){dt\over t}
  \\
  &\le&  C 2^{2(j_0-j)N}2^{-2 \ell N} \delta \sum_{\mm\not\in \mathbf M_0 }  (1+|x_m|)^{-\al}
   \big\| \bchi_{\widetilde {Q_{\mathbf m}}} \varphi_k(\!\sqrt{H}\,)  f\big] \big\|_{2}^2
  \\
   &\leq & C 2^{2(j_0-j)N}2^{-2 \ell N} \delta \sum_{m}
    \int_{{\widetilde {Q_{\mathbf m}}}}
 | \bchi_{\widetilde {Q_{\mathbf m}}} \varphi_k(\!\sqrt{H})  f(x)|^2 (1+|x|)^{-\al}dx
  \\
	 &\leq & C 2^{2(j_0-j)N}2^{-2 \ell N}  \delta \int_{\RN} |\varphi_k (\!\sqrt{H})f(x)|^2  (1+|x|)^{-\al}dx.
  \end{eqnarray*}
Finally, recalling \eqref{e.I3} and taking sum over $\ell$
 yields the estimate \eqref{e.est-I3}.

  Therefore,  recalling $\delta\leq \delta A_{\alpha, n}^{\epsilon}(\delta) $,  we combine  the  estimates
    \eqref{ecccc}, \eqref{e.est-I2}, and \eqref{e.est-I3}  with  \eqref{mmmmm}
 to   obtain
 \Be
  \label{e.final}
  \int_{2^{k-1}}^{2^{k+2}}
  E^{k, j}(t){dt\over t}
    \le C_N 2^{(j_0-j) N} \delta A_{\alpha, n}^{\epsilon}(\delta) \int_{\RN}
| \varphi_k(\!\sqrt{H}\,)  f (x)|^2\frac{dx} {(1+|x|)^{\alpha}}
  \Ee
for any $N>0$ if $j>k$.

\subsubsection{\bf Case 2:   $j\leq k$.} In this case,  the side length of each $Q_{\mathbf m}$ is less than $4$.
 Thus, \eqref{e4.299} holds for any $\mm\in \mathbb Z^n$.
Thus, the same argument in the proof  of \eqref{e.est-I3} works without modification.  Similarly as before,  we get
  \begin{eqnarray}
  \label{e.final1}
\int_{2^{k-1}}^{2^{k+2}}
 E^{k, j}(t){dt\over t}
    &\leq& C_N 2^{(j_0-j)N}   \delta \int_{\RN}
| \varphi_k(\!\sqrt{H}\,)  f (x)|^2 \frac{dx} {(1+|x|)^{\alpha}}
  \end{eqnarray}
for any $N>0$, which is stronger than \eqref{e.final}.

 \subsubsection{\bf Completion of the proof of  \eqref{et121}.}
Finally, we are in position to complete  the proof of  \eqref{et121}.  By the estimates \eqref{e.final} and \eqref{e.final1}
we now have the estimate \eqref{e.final} for any $j\ge j_0$ and $k$.  Putting  \eqref{e.final}  in the right hand side of \eqref{e4.28}
and then taking sum over $j$,   we obtain  	
\[
 \int_{\RN}| \mathfrak S^{h}_{\delta}f(x)|^2 (1+|x|)^{-\alpha}dx
 \le      C\delta A_{\alpha, n}^{\epsilon}(\delta)
         \int_{\RN} \sum_k| \varphi_k (\!\sqrt{H}\,)f(x)|^2 (1+|x|)^{-\al}dx.
     \]
    Using  Proposition~\ref{prop2.5}   we get the estimate \eqref{et121} and this completes the proof of Lemma~\ref{le4.2}.

 \begin{rem} Let us generalize the square functions  by setting
 \label{re:aboutdelta}
  \begin{align*}
  \mathfrak S_\tau f(x) & =: \Big( \int_{1/2}^{\tau} \Big|\phi\Big(\delta^{-1}\Big(1-{ {H}\over t^2} \Big)\Big)f(x)\Big|^2
   {dt\over t}  \Big)^{1/2},  \\
    \mathfrak S^\tau f(x)
   &=: \Big( \int_{\tau}^\infty \Big|\phi\Big(\delta^{-1}\Big(1-   { {H}\over t^2} \Big) \Big)f(x)\Big|^2
   {dt\over t}\Big)^{1/2}
\end{align*}
for $\tau\gg 1$. By  examining  the proofs in the above  one can obtain  the bounds
on $\mathfrak S_\tau$ and $\mathfrak S^\tau$ in the space  $L^2(\RN, \, (1+|x|)^{-\alpha})$.
In fact, it is not difficulty to see that,  for $n\ge 2$ and $\alpha>1$,
\begin{align*}
\left\|\mathfrak S_\tau\right\|_{L^2(\RN, \, (1+|x|)^{-\alpha}) \to L^2(\RN, \, (1+|x|)^{-\alpha})}
&\le  C\delta\tau^{\frac{\alpha+1}2}
\end{align*}
and
\begin{align*}
   \left\|\mathfrak S^\tau\right\|_{L^2(\RN, \, (1+|x|)^{-\alpha}) \to L^2(\RN, \, (1+|x|)^{-\alpha})}
 &\le   C\delta^{1-\frac{\alpha} 2}\tau^{\frac{1-\alpha}2}.
 \end{align*}
Optimization between these two estimates gives the choice $\tau=\delta^{-\frac12}$.
 \end{rem}

\section{Sharpness of summability indices}
\label{sec4}
\setcounter{equation}{0}

  In  this section  we show the summability index for a.e. convergence in Theorem  ~\ref{th1.1} and  that  for boundedness of
  $S_*^\lambda(H)$  on $L^2(\RN, (1+|x|)^{-\alpha})$  in Theorem  ~\ref{th1.2}
  are sharp  up to endpoint.  For the purpose we  only have to prove the following two propositions.

  \begin{prop}\label{nec-1}  Le  $n\ge 2$ and $2n/(n-1)<p<\infty$.   If
 $$\sup_{R>0} |S_R^{\lambda}(H) f|<\infty \ \ \ \  {\rm a.e }
 $$
 for all $f\in L^p(\mathbb R^n)$, then
	 we  have   $\lambda\geq  \lambda(p)/2$.
\end{prop}

 Since we are assuming $\lambda\ge 0$,  Proposition \ref{nec-1} shows the summability index in Theorem \ref{th1.1} is sharp up to endpoint.

  \begin{prop}\label{nec-2} Let   $0\leq \alpha<n$.  Suppose that
\begin{eqnarray}\label{e7.1}
\sup_{R>0}\left\|S_R^{\lambda}(H)\right\|_{L^2\big(\RN, \, (1+|x|)^{-\alpha}) \to L^2(\RN, \, (1+|x|)^{-\alpha}\big)}  < \infty.
\end{eqnarray}
Then, we have $
\lambda \ge  \max \{0,\frac{\alpha-1}{4}\}.
$
\end{prop}

This clearly implies the necessity  part of Theorem~\ref{th1.2}  because
\[
\sup_{R>0}\left\|S_R^{\lambda}(H)f\right\|_{L^2(\mathbb R^n,(1+|x|)^{-\al})}
\le \|\sup_{R>0}|S_R^{\lambda}(H)f|\|_{L^2(\mathbb R^n,(1+|x|)^{-\al})}.
\]

 \subsection{Proof of Proposition \ref{nec-1}}
  To prove  Proposition \ref{nec-1} we use  a consequence of the  Nikishin-Maurey theorem (see for example, \cite {de,  GR,  St2}).
  The following can be deduced from  Proposition 1.4  and  Corollary 2.7  in \cite[Ch. VI]{GR}.

\begin{thm}  \label{Nikishin}
Let  $1\leq p<\infty$ and $(X, \mu)$   be a  $\sigma$-finite measure space.
Suppose that  $\{T_m\}_{m\in{\mathbb N}}$
is a sequence of  linear   operators   continuous from $L^p(X, \mu)$ to $L^0(X, \mu)$ of  all measurable
functions    on $(X, \mu)$ (with the topology of convergence in measure).   Assume that
$
T^{\ast}f= \sup_m |T_mf|<\infty
$
a.e.  whenever $f\in L^p(X, \mu)$, then  there exists  a measurable  function  $w>0$   a.e.  such that
$$
\int_{ \{x: |T^{\ast}f(x)|>\alpha \} } w(x)d\nu(x) \leq  C \left( \|f\|_p\over \alpha\right)^q, \quad  \alpha>0
$$
for all $f\in L^p(X, \mu)$ with $q=\min(p,2)$.
\end{thm}

In the remaining part of this subsection  we mainly work with radial functions. For a given function $g$ on $[0,\infty)$, we set $g_{rad}(x):=g(|x|)$.
Clearly, $g_{rad}\in  L^p(\RN)$ if and only if $g\in
 L^p([0,\infty),r^{n-1}dr)$. For the   proof of Proposition~\ref{nec-1}  we first show Proposition~\ref{th6.1}.

\begin{prop}\label{th6.1} Let   $ 2\leq p<\infty$ and
$w$ be a measurable function on $(0, \infty)$ with $w> 0$ almost everywhere.   Suppose that
\begin{eqnarray}\label{ne7.1}
\sup_{R>0}\sup_{\alpha>0}\alpha\left(\int_{ \big\{x\in \RN: \, |S_R^{\lambda}(H)f(x)|>\alpha \big\} } w(|x|)dx \right)^{1/2}  \le C\|f\|_{L^{p}(\RN)}
\end{eqnarray}
for all radial functions $f\in L^p(\RN)$. Then, we have $
\lambda \ge n(1/2-1/p)/2-1/4.
$
\end{prop}

\begin{proof}[Proof  of Proposition~\ref{nec-1}]
Let $R_m$ be any enumeration of the rational numbers in $(0,\infty)$. Then we have $S_*^{\lambda}(H)f(x)
 =\sup_{m}|S_{R_m}^{\lambda}(H)f(x)|$, which follows since,  for every $x$, $S_{R}^{\lambda}(H)f(x)$ is right continuous in $R$, $0\leq R<\infty$.

We now restrict the operator $S_{R_m}^{\lambda}(H)$ to the set  of radial functions. Since  $S_R^{\lambda}(H)f$
  is  a radial function if  $f$ is radial, we may view  $S_{R_m}^{\lambda}(H)$  as  an operator from $L^p([0,\infty),r^{n-1}dr)$ to itself.
  \footnote{Though uniform boundedness of $S^{\lambda}_R(H)$ remains open,   $S^{\lambda}_R(H)$ is clearly  bounded on $L^p$
   since there are finitely many Hermite functions appearing in  $S^{\lambda}_R(H)$.}   More precisely,
  let us denote  by $T_m$ the mapping
  \[ g \mapsto S_{R_m}^{\lambda}(H)g_{rad} \]
for $g\in  L^p([0,\infty),r^{n-1}dr)$.  Since $\sup_{R>0} |S_R^{\lambda}(H) f|<\infty $  a.e.  for all $f\in L^p(\mathbb R^n)$
by assumption, we clearly have
$$
\sup_{m}|T_m g(r)|<\infty,\,\,a.e.\,\,r\in[0,\infty).
$$
for all $g\in  L^p([0,\infty),r^{n-1}dr)$. Now   we take $(X,\mu) =([0,\infty),r^{n-1}dr)$.
Clearly, each operator  $T_m$ is  continuous $L^p(X,\mu)$ to itself  and so is  $T_m$ from $L^p(X,\mu)$
 into $L^0(X,\mu)$.  Then it follows from Theorem~\ref{Nikishin} that  there exists a weight function $ w>0$ such that
 $g\mapsto \sup_{m}|T_m g|$ is bounded from $L^p([0,\infty),r^{n-1}dr)$ to
  $L^{2,\infty}([0,\infty), w(r)r^{n-1}dr)$. So, we get  a weight $w$ which satisfies \eqref{ne7.1}   for all radial functions $f\in L^p(\RN)$ because
  $S_*^{\lambda}(H)f(x)
 =\sup_{m}|S_{R_m}^{\lambda}(H)f(x)|$ for any $f\in L^p(\RN)$. Therefore, by Proposition~\ref{th6.1},    we   conclude  that $\lambda \ge   n(1/2-1/p)/2-1/4$.
 \end{proof}

It remains to prove Proposition~\ref{th6.1}. For this we make use of estimates for
 the Hermite and  Laguerre functions.
Recall that  the Laguerre polynomials of type $\alpha$ are defined by the formula (see \cite[1.1.37]{Th3}):
$$
e^{-x} x^{\alpha} L_k^{\alpha}(x) ={1\over k!} {d^k\over d x^k} \left( e^{-x} x^{k+\alpha}\right), \quad \alpha>-1.
$$
Define
\Be
\label{laguerre}
{\mathscr L}_k^\alpha(x)=\left(\frac{\Gamma(k+1)}{\Gamma(k+\alpha+1)}\right)^{\frac{1}{2}}e^{-\frac{x}{2}}x^{\frac{\alpha}{2}}L_k^\alpha(x).
\Ee
The functions $\{{\mathscr L}_k^\alpha\}$ form an orthonormal family in $L^2({\mathbb R}_+, dx)$ where ${\mathbb R}_+=(0, \infty)$.
Recall that $P_k$ is the Hermite spectral projection operator defined by  \eqref{e1.5} and set
\[\mathfrak L_k^n(r):=L_k^{n/2-1}(r^2)e^{-\frac{1}{2}r^2}.\]

\begin{lemma}
 \cite[Corollary 3.4.1]{Th3}
\label{len6.1}
If  $f(x)=f_0(|x|)$ on $\RN$, then $P_{2k+1}(f)=0$ and
$$
P_{2k}(f)=R_k^{n/2-1}(f_0)\mathfrak L_k^n(r),
$$
where
$$
R_k^{n/2-1}(f)= \frac{2\Gamma(k+1)}{\Gamma(k+n/2)}\int_0^\infty f(r)\mathfrak L_k^n(r)r^{n-1}dr.
$$
\end{lemma}

\begin{lemma}
\cite[(i) in Lemma 1.5.4]{Th3}
\label{len6.2}
Let $\alpha+\beta>-1$, $\alpha>-2/q,$ and $1\leq q\leq 2$. Then, if  $k$ is large enough,  for $\beta<2/{q}-1/{2}$ we have
$$
\|{\mathscr L}_k^{\alpha+\beta}(x)x^{-\beta/2}\|_{L^q({\mathbb R}_+)}\sim k^{\frac{1}{q}-\frac{1}{2}-\frac{\beta}{2}}.
$$
\end{lemma}

\begin{lemma}\label{len7.1} Let $w$ be a measurable function on $(0, \infty)$   with $w> 0$ almost everywhere.
 If  $k\in{\mathbb N}$ is large enough, we have
\begin{align}\label{ne7.00}
\sup_{\beta>0}\beta\Big(\int_{ \big\{x\in \RN: \, |\mathfrak L_k^n(|x|)|>\beta \big\} } w(|x|)dx \Big)^{1/2}\geq  C_1D_k k^{-1/4}
\end{align}
with  a constant $C_1>0$ independent of $k$ where $D_k=\left({\Gamma(k+n/2)}/{\Gamma(k+1)}\right)^{1/2}$.
\[\]
\end{lemma}

To prove Lemma~\ref{len7.1},
we make use of  the following asymptotic property of the Laguerre functions (see \cite[7.4, p.453]{Mu}):
\Be
\label{asymp}
{\mathscr L}_k^{\alpha}(r)= \frac{(2/\pi)^{1/2}}{(\nu r)^{1/4}}
\left(\cos\left((\nu r)^{1/2}-\frac{\alpha\pi}{2}-\frac{\pi}{4}\right)+O(\nu^{-1/2}r^{-1/2})\right),
\Ee
where $\nu=4k+2\alpha+2$ and  $1/\nu\leq r\leq 1$.


\begin{proof}
Let us set
\[  E(k):=\left\{r \in \big[\frac{1}{2},1 \big]:\left|\cos\left(\sqrt \nu r-\frac{\alpha\pi}{2}-\frac{\pi}{4}\right)\right|\geq \frac{\sqrt 2}2\right\}.\]
We claim that there exists a constant $C_\ast>0$, independent of $k$, such that
\begin{eqnarray}\label{ne7.000}
|E(k)|\geq C_\ast.
\end{eqnarray}
Assuming this for the moment, we proceed to show \eqref{ne7.00}.  By \eqref{laguerre} and \eqref{asymp} we have
$$
|\mathfrak L_k^n(r)|=|{\mathscr L}_k^{n/2-1}(r^2)r^{-(n/2-1)}D_k|\geq CD_kk^{-1/4}
$$
for all $r\in E(k)$.
On the other hand, since $w(r)> 0$ a.e. $r\in [1/2,1]$, there exists a subset $F$ of $[1/2,1]$
and a constant $c_0>0$ such that $|F|>1/2-C_\ast/2$ and $w(r)\geq c_0$  for all $r\in F$.
Then, we have $|E(k)\cap F|\geq |E(k)|+|F|-|[1/2,1]|\geq C_\ast/2$ and
\[
 \sup_{\beta>0} \beta^2\int_{\big\{x\in \RN:\, |\mathfrak L_k^n(|x|)|>\beta \big\} } w(|x|)dx
\ge
 \sup_{\beta>0} \beta^2\int_{ \big\{E(k)\cap F:\, |\mathfrak L_k^n(r)|>\beta  \big\} } w(r)r^{n-1}dr.
\]
Since $w(r)\geq c_0$ for $r\in F$, the  left hand side of the above is bounded below by
\[  c_0(1/2)^{n-1}\sup_{\beta>0} \beta^2\int_{ \big\{E(k)\cap F: \,|\mathfrak L_k^n(r)|>\beta \big\} } dr.  \]
Particularly, taking $\beta=CD_kk^{-1/4}/2$ and using the fact that  $ |\mathfrak L_k^n(r)|\geq CD_kk^{-1/4}$ for $r\in E(k)$ and  $|E(k)\cap F|\geq C_\ast/2$,
we see
\begin{eqnarray*}
\sup_{\beta>0} \beta^2\int_{\big\{x\in \RN:\, |\mathfrak L_k^n(|x|)|>\beta \big\} } w(|x|)dx
&\ge& c_0(1/2)^{n-1}(CD_kk^{-1/4}/2)^2|E(k)\cap F|\\
&\geq& c_0 (1/2)^{n+2} C^2 C_\ast D_k^2k^{-1/2},
\end{eqnarray*}
which implies that \eqref{ne7.00} holds for $C_1=c_0 (1/2)^{n+2} C^2C_\ast$.

It now remains to prove \eqref{ne7.000}, which is rather obvious. However, we include a proof for the convenience of the reader.
Note that if there exists $m\in \NN$ such that
$$
2m\pi-\frac{\pi}{4}\leq \nu^{1/2}r-\frac{\alpha\pi}{2}-\frac{\pi}{4}\leq 2m\pi+\frac{\pi}{4},
$$
that is,
$
\nu^{-\frac12} (2m\pi+\alpha\pi/2)\le r\le \nu^{-\frac12}(2m\pi+(\alpha+1)\pi/2),
$
then $\cos(\nu^{1/2}r-\frac{\alpha\pi}{2}-\frac{\pi}{4})\geq \sqrt 2/2$. 
 Then there are at  least $[\frac{\sqrt \nu}{4\pi}]$ intervals $\nu^{-\frac12}[ (2m\pi+\alpha\pi/2),
 (2m\pi+\alpha\pi/2)+ 2\pi ]$  and so $[\frac{\sqrt \nu}{4\pi}]$ intervals $\nu^{-\frac12}[ (2m\pi+\alpha\pi/2),
 (2m\pi+(\alpha+1)\pi/2)]$ included in $[1/2,1]$. Thus,
$$
|E(k)|\geq \Big[\frac{\sqrt \nu}{4\pi}\Big]\frac{\pi}{2\sqrt v}\geq 1/16.
$$
So,   \eqref{ne7.000} holds for $C_\ast=1/16$. This completes the proof of Lemma~\ref{len7.1}.
\end{proof}

\begin{lem}\label{len7.3}
Let $1\leq q\leq 2$. Then we have the estimate
\begin{eqnarray}\label{ne7.13}
\|\mathfrak L^n_k\|_{L^{q}([0, \infty), \, r^{n-1}dr) }\sim D_k k^{n(1/q-1/2)/2}.
\end{eqnarray}
\end{lem}

\begin{proof}  By \eqref{laguerre} we note that
$
|\mathfrak L^n_k(r)|
=D_k |{\mathscr L}_k^{n/2-1}(r^2)r^{-(n/2-1)}|.
$
We take $\alpha=2(n/2-1)/q$ and $\beta=2(1/2-1/q)(n/2-1)$ in   Lemma~\ref{len6.2} to obtain
\begin{align*}
\int_0^\infty|\mathfrak L_k^n(r)|^q r^{n-1}dr
&= D_k^q \int_0^\infty |{\mathscr L}_k^{n/2-1}(r^2)r^{-(n/2-1)}|^q r^{n-1}dr\\
&\sim D_k^pk^{\frac{nq}{2}(\frac{1}{q}-\frac{1}{2})}.
\end{align*}
 The proof of    Lemma~\ref{len7.3} is complete.
\end{proof}

In order to prove Proposition~\ref{th6.1},
we use  the
 distributions $\chi_-^\nu$   (see \cite{Ho1}) which is defined by
\begin{eqnarray}\label{e9.11}
\chi_-^{\nu}=\frac{x_-^{\nu}}{\Gamma({\nu}+1)},\ \ \ \   {\rm Re}\, {\nu}>-1,
\end{eqnarray}
where $\Gamma$ is the Gamma function and
$
x_-=|x|$ if $ x \le 0$ and $x_-=0$ if $ x > 0$.
For $\Re {\nu} > -1$,
the distribution $\chi_-^{\nu}$ is clearly well defined.
see \cite[p. 308]{GP} or  \cite{CHS, DOS}.
For $\Re {\nu} \le -1$, $\chi_-^{\nu}$ can extended by analytic continuation
(see, for example, \cite[Ch III, Section 3.2]{Ho1}). For compactly supported function $F$ such that $\supp F \subset [0,\infty)$,
 the Weyl fractional derivative of $F$ of order $\nu$ is given by the formula
\begin{equation}\label{e9.13} F^{(\nu)}=F*{\chi}^{-\nu-1}_-, \ \ \ \nu\in{\mathbb C}.
\end{equation}
Since $F= F^{(\nu)}*{\chi}^{\nu-1}_-$,  we may write $F(H)=\frac{1}{\Gamma(\nu)}\int_0^\infty F^{(\nu)}(t)
(t-{H})_+^{\nu-1}dt$ if $F$ has compact support in $[0,\infty)$. Thus, it follows that,  for every $\nu\geq 0,$
\Be \label{e9.14}
F(H)
= \frac{1}{\Gamma(\nu)}\int_0^\infty F^{(\nu)}(R)\, R^{\nu-1}\, S_{\sqrt{R}}^{\nu-1}(H)dR
\Ee
for all   $F$ compactly supported in $[0, \infty)$.

Now  we are ready to prove Proposition~\ref{th6.1}.
\hspace{-4cm}
\begin{proof}[Proof of Proposition ~\ref{th6.1}]
 Let us set $\widetilde w(x)=w(|x|)$.
Using  \eqref{e9.14}, we  see that  $ F(H) f(x)$ is qual to
  \[ \frac{1}{\Gamma(\lambda+1)}
  \int_0^\infty F^{(\lambda+1)}(R) R^{\lambda}S_{\sqrt{R}}^{\lambda}(H) f(x) dR.
\]
Since $L^{2,\infty}$ is normable, by Minkowski's inequality  and  the assumption \eqref{ne7.1}  we have
 \begin{eqnarray*}
  \|F(H) f\|_{L^{2,\infty}(\widetilde wdx, \RN)}
 &\leq&   C\sup_{R>0}\|S_R^{\lambda}(H)f\|_{L^{2,\infty}(\widetilde wdx, \, \RN) }
                  \int_0^\infty  |F^{(\lambda+1)}(s) |s^{\lambda}ds
                  \\
 &\leq&   C \,  \|f\|_{L^p(\RN) }  \int_0^\infty  |F^{(\lambda+1)}(s) |s^{\lambda}ds
\end{eqnarray*}
for $F$ compactly supported in $[0,\infty)$.
Let $\eta$ be a non-negative smooth function  such that $\eta(0)=1$ and  $\supp
\eta \subset [-1,1]$.  Taking $F=\eta(\cdot-2k)$ in the above estimate, we get
\Be  \label{ne7.4}
\|\chi_{[2k,2k+1)}(H)f\|_{L^{2,\infty}(\widetilde wdx, \, \RN) }\leq  Ck^\lambda \|f\|_{L^p(\RN) }
\Ee
because  $
\int_0^\infty  |\eta^{(\lambda+1)}(s-k) |s^{\lambda}ds \sim k^{\lambda},
$
and  $\eta(H-2k)f=\chi_{[2k,2k+1)}(H)f=P_{2k}f.$

Now, let us set \[ f_k(r):= \operatorname{sign} (\mathfrak L_k^n (r)) | \mathfrak L_k^n (r)|^{1/(p-1)}.\]
Then, from Lemma~\ref{len6.1} and Lemma~\ref{len7.1}, it follows that
we obtain
 \begin{eqnarray*}
 \|\chi_{[2k,2k+1)}(H)f_k(|\cdot|)\|_{L^{2,\infty}(\widetilde wdx, \, \RN)}
 &=&R_k^{n/2-1}(f_k)\|\mathfrak L_k^{n}(|\cdot|)\|_{L^{2,\infty}(\widetilde wdx, \, \RN)}
 \\
 &\geq&CD_k^{-1} k^{-1/4}\int_0^\infty f_k(r)\mathfrak L_k^n(r) r^{n-1}dr.
\end{eqnarray*}
By our choice of $f_k$ it is clear that $\int_0^\infty f_k(r)\mathfrak L_k^n(r) r^{n-1}dr
=\|\mathfrak L_k^n\|^{p'}_{L^{p'}([0, \infty), \, r^{n-1}dr)} $ $=\|f_k\|_{L^p([0, \infty), \, r^{n-1}dr)}
\|\mathfrak L_k^n\|_{L^{p'}([0, \infty), \, r^{n-1}dr)}$. Thus, using Lemma \ref{len7.3}, we get
 \[
 \|\chi_{[2k,2k+1)}(H)f_k(|\cdot|)\|_{L^{2,\infty}(\widetilde wdx, \, \RN)}
 \ge Ck^{n(1/2-1/p)/2-1/4} \|f_k(|\cdot|)\|_{L^p(\RN)} .
\]
Then we  combine this with
\eqref{ne7.4} where we take  $f=f_k(|\cdot|)$
to obtain
 \begin{eqnarray} \label{ne7.44}
k^{n(1/2-1/p)/2-1/4} \|f_k(|\cdot|)\|_{L^p(\RN)}
 \leq  Ck^{\lambda} \,  \|f_k(|\cdot|)\|_{L^p(\RN)}
\end{eqnarray}
with $C$ independent of $k$.
Obviously
$
0<\|f_k\|^p_{L^p(\RN)}<\infty.
$
Thus,  \eqref{ne7.44} implies $k^{n(1/2-1/p)/2-1/4}
 \leq  Ck^{\lambda}$ with $C$ independent of $k$.
 Letting  $k$ tend to infinity, we get
$
\lambda\geq n(1/2-1/p)/2-1/4
$
as desired. \end{proof}

\subsection{Proof of Proposition ~\ref{nec-2}}

The proof of Proposition ~\ref{nec-2} is based on the following  weighted estimates of the normalized Hermite functions.

\begin{lemma}\label{le7.1} Let $\alpha\geq 0$.  Then, if  $k\in{\mathbb N}$ is large enough, we have
\begin{align}\label{e7.0}
\int_{-\infty}^\infty h^2_{k}(x)(1+&|x|)^\alpha dx  \geq  C k^{\alpha/2},
\\
\label{e7.00}
\int_{-\infty}^\infty h^2_{k}(x)(1+|x|)^{-\alpha} & dx \geq  C \max\{k^{-\alpha/2},k^{-1/2}\}.
\end{align}
\end{lemma}

To prove the lower bounds \eqref{e7.0} and \eqref{e7.00}, we  make  use  of
the  following  asymptotic property  of the Hermite function (see \cite[1.5.1, p. 26]{Th3}):
\Be\label{eee51}
h_{k}(x)= \left(\tf{2}{\pi}\right)^{\frac{1}{2}} \left(N-x^2\right)^{-\frac{1}{4}} \cos\left( \frac{N(2\theta-\sin\theta )-\pi}{4} \right)
+O\left(N^{\frac{1}{2}}(N-x^2)^{-\frac{7}{4}}\right),
\Ee where   $N=2k+1$, $0\leq x\leq N^{ {1}/{2}}-N^{- {1}/{6}} $ and $\theta=\arccos(xN^{- {1}/{2}})$.

\begin{proof}
We begin with showing   that there exists a constant $C>0$ such that, for any large $N,$
\begin{eqnarray}\label{e7.000}
|E(N)|\geq C \sqrt N,
\end{eqnarray}
where
\[
 E(N):=  \left\{x \in \left[\frac{\sqrt N}{2},\frac{\sqrt N}{\sqrt 2}\right]:
\cos\left( \frac{N(2\theta-\sin\theta )-\pi}{4} \right)\geq {\sqrt 2\over 2}\right\}\,,
\]
and $  \theta=\arccos(xN^{- {1}/{2}})$.
For  \eqref{e7.000},  it is enough to show
\Be \label{e7.0001}
\left|\left\{t \in \left[\frac{1}{2},\frac{1}{\sqrt 2}\right]:\cos\left( \frac{N(2\tilde \theta-\sin\tilde\theta )-\pi}{4} \right)\geq
{\sqrt 2\over 2} \right\}\right|\geq C, \ \ \ \  \tilde\theta=\arccos(t)
\Ee
with $C$  independent of $N$,  which is equivalent to
\eqref{e7.000}  as is easy to see by change of variables.
In order to show \eqref{e7.0001}, we make change of variables  $y=2\tilde \theta-\sin\tilde\theta$.
The condition
$t \in [ {1}/{2}, {1}/{\sqrt 2}]$ implies that $\tilde\theta\in [ {\pi}/{4}, {\pi}/{3}]$
and $y\in [ {\pi}/{2}- {\sqrt 2}/{2}, {2\pi}/{3}- {\sqrt 3}/{2}]$.
We note that
$
-\frac{3\sqrt 2}{2}<\frac{d y}{dt} =-\frac{2-t}{\sqrt {1-t^2}}<-\frac{2-\sqrt 2}{2}
$
for $t \in [ {1}/{2}, {1}/{\sqrt 2}]$.
So,  \eqref{e7.0001} follows if we show  that there exists a constant $C>0$ independent on $N$ such that
\begin{eqnarray*}
\left|\left\{y \in \left[\frac{\pi}{2}-\frac{\sqrt 2}{2},\frac{2\pi}{3}-\frac{\sqrt 3}{2}\right]:\
 \ \cos\left( \frac{Ny-\pi}{4} \right)\geq {\sqrt 2\over 2}\right\}\right|\geq C,
\end{eqnarray*}
but this is clear from an elementary computation.

Once we have  \eqref{e7.000},  the desired estimate  \eqref{e7.0} follows because
$h_{k}(x)\geq C N^{-1/4}, x\in E$ by \eqref{eee51}. Clearly, we  also  have the following estimate
$\int_{-\infty}^\infty h^2_{k}(x)(1+|x|)^{-\alpha} dx
\geq Ck^{-\alpha/2}$. To complete the proof it remains  to show
\begin{eqnarray*}
\int_{-\infty}^\infty h^2_{k}(x)(1+|x|)^{-\alpha} dx
\geq Ck^{-1/2}.
\end{eqnarray*}
In the similar manner as before   it is easy to  show (see also \cite[Lemma 3.4]{BR}) that
\begin{eqnarray*}
 \left|\left\{x \in [0,1]:\cos\left( \frac{N(2\theta-\sin \theta)-\pi}{4} \right)\geq {\sqrt2\over 2}\right\}\right|\geq C
\end{eqnarray*}
with $C$  independent of $N$.
 Combining this with \eqref{eee51} we get
$\int_{0}^1 h^2_{k}(x) dx \geq C N^{-1/2} \geq Ck^{-1/2}$ and hence the desired estimate.
\end{proof}

\begin{lem}\label{le5.3} Let $\alpha\ge 0$. Then, for  all $f\in L^2$,  we have the estimate
\begin{eqnarray}\label{e7.13}
 k^{\alpha/4} \|\chi_{[k,k+1)}(H)f\|_2\le C\|\chi_{[k,k+1)}(H) f\|_{L^2(\RN, \, (1+|x|)^{\alpha}) }.
\end{eqnarray}
\end{lem}

\begin{proof}

As in the proof of  Lemma \ref{le3.1}, we decompose $f=\sum_{j=1}^n f_j$ such that   $f_1,\dots, f_n$ are orthogonal to
each other and, for $1\le i\le n$,   $\mu_i\geq |\mu|/n$ whenever
$\langle  f_i, \Phi_\mu \rangle\neq 0$ and, additionally,  the supports
of the maps $\mu\to \langle  f_i, \Phi_\mu \rangle$ are mutually disjoint.
So,  we  have
\[ \|\chi_{[k,k+1)}(H)f\|_2^2= \sum\limits_{j=1}^n  \|\chi_{[k,k+1)}(H)f_j\|_2^2,\]
and  there is a $j\in \{1,\cdots, n\} $
such that $  \|\chi_{[k,k+1)}(H)f_j\|_2^2\ge n^{-1}  \|\chi_{[k,k+1)}(H)f\|_2^2$.
Thus, it is sufficient for \eqref{e7.13} to show that
\begin{equation}
\label{e5.weight}
k^{\alpha/2} \|\chi_{[k,k+1)}(H)f_j\|_2^2\le C\|\chi_{[k,k+1)}(H) f\|^2_{L^2(\RN, \, (1+|x|)^{\alpha}) }.
\end{equation}
Without loss of generality, we may assume   $j=1$.

We proceed to show \eqref{e5.weight} for $j=1$.
 Since $(1+|x|)^{\alpha}\ge (1+|x_j|)^{\alpha}$, we have
\[
\|\chi_{[k,k+1)}(H)f\|^2_{L^2(\RN, \, (1+|x|)^{\alpha}) }
\! \ge \!  \int_{\RN}\sum_{2|\mu|+n=k}\sum_{2|\nu|+n=k}  \!\!\!\!\!c(\mu) \overline{c(\nu)} \Phi_\mu(x){\Phi_\nu(x)}  (1+|x_1|)^\alpha  dx,
\]
where $
 c(\mu)=\langle  f, \Phi_\mu \rangle$. Clearly the right hand side of the above is equal to
\[
 \sum_{2|\mu|+n=k}\sum_{2|\nu|+n=k} c(\mu) \overline{c(\nu)}\int_{-\infty}^{\infty}h_{\mu_1}(x_1)\overline{h_{\nu_1}(x_1)} (1+|x_1|)^\alpha dx_1
\prod_{i=2}^n \langle h_{\mu_i},  h_{\nu_i}  \rangle.
\]
By orthonormality of the Hermite functions and the relation $2|\mu|+n=k=2|\nu|+n$  this is again identical to
\[
\sum_{2|\mu|+n=k} |c(\mu)|^2 \int_{-\infty}^\infty h^2_{\mu_1}(x_1)(1+|x_1|)^\alpha dx_1 .
\]
Therefore, using  \eqref{e7.0} in Lemma~\ref{le7.1},  we have
\begin{eqnarray*}\label{e7.10}
\|\chi_{[k,k+1)}(H)f\|^2_{L^2(\RN, \, (1+|x|)^{\alpha}) }
\geq \sum_{2|\mu|+n=k} |c(\mu)|^2 \mu_1^{\alpha/2}.
\end{eqnarray*}
 This yields the desired estimate  \eqref{e5.weight}   for $j=1$ because  $\mu_1\geq |\mu|/n$ whenever $\langle  f_i, \Phi_\mu \rangle\neq 0$. Indeed,
it follows from the construction of $f_i$ and $c(\mu,i)$ that
   $$
  \sum_{2|\mu|+n=k} |c(\mu)|^2 \mu_1^{\alpha/2}\geq \sum_{2|\mu|+n=k} |\langle  f_1, \Phi_\mu \rangle|^2 \mu_1^{\alpha/2}\sim
 k^{\alpha/2}  \|\chi_{[k,k+1)}(H)f_1\|_2^2.
 $$
 Therefore, we get \eqref{e5.weight}.
\end{proof}

\begin{proof}[Proof of Proposition~\ref{nec-2}]
Since we are assuming that \eqref{e7.1} holds,  by duality we have the equivalent estimate
\[\sup_{R>0}\left\|S_R^{\lambda}(H)\right\|_{L^2(\RN, \, (1+|x|)^{\alpha}) \to L^2(\RN, \, (1+|x|)^{\alpha})}  < \infty.\]
We combine this and \eqref{e9.14} with $\nu=\lambda+1$ to obtain
\begin{align*}
\|F(H)f\|_{L^2(\RN, \, (1+|x|)^{\alpha}) }
&\leq  C\sup_{R>0}\|S_R^{\lambda}(H)f\|_{L^2(\RN, \, (1+|x|)^{\alpha}) }
          \int_0^\infty  |F^{(\lambda+1)}(s) |s^{\lambda}ds
           \\
 &\leq  C
        \|f\|_{L^2(\RN, \, (1+|x|)^{\alpha}) }  \int_0^\infty  |F^{(\lambda+1)}(s) |s^{\lambda}ds
 \end{align*}
for $F$ compactly supported  in  $\supp F \subset [0,\infty)$.  Similarly as before,  we take $F(t)=\eta(t-k)$ in the above where
$\eta$ is a non-negative smooth function with $\eta(0)=1$ and  $\supp
\eta \subset [-1,1]$. Then, since  $\int_0^\infty  |\eta^{(\lambda+1)}(s-k) |s^{\lambda}ds \sim k^{\lambda}$ and  $\eta(H-k)=\chi_{[k,k+1)}(H),$ it follows that
\begin{eqnarray} \label{e7.4}
\|\chi_{[k,k+1)}(H)f\|_{L^2(\RN, \, (1+|x|)^{\alpha}) }\leq  Ck^\lambda \|f\|_{L^2(\RN, \, (1+|x|)^{\alpha}) }.
\end{eqnarray}

We now consider specific functions $g_k$, $G_k$ which are given   by
\[ g_k(x)=h_{k}(x_1)h_0(x_2)\cdots h_0(x_n), \quad G_k(x)=g_k(x)(1+|x|)^{-\alpha},\]
and claim that
\Be \label{e7.44}
\left\|\chi_{[k,k+1)}(H)G_k\right\|_{L^2(\RN, \, (1+|x|)^{\alpha}) }
 \leq  Ck^\lambda \min\big\{k^{\alpha/4},k^{1/4}\big\} \,  \|\chi_{[k,k+1)}(H)G_k\|_2
\Ee
with $C$ independent of $k$.
Indeed, since $\|g_k\|_2=1$, we have  $\|\chi_{[k,k+1)}(H)G_k\|_2\ge    \langle\chi_{[k,k+1)}(H) G_k,  g_k \rangle =\langle G_k, \chi_{[k,k+1)}(H)g_k \rangle$.
Thus, noting that  $ \chi_{[k,k+1)}(H)g_k=g_k$ from our choice of $g_k$ and $g_k(x)(1+|x|)^{-\alpha/2}=G_k(x)(1+|x|)^{\alpha/2}$,  we get
\[
\|\chi_{[k,k+1)}(H)G_k\|_2 \ge \langle G_k, g_k \rangle
= \|(1+|x|)^{\alpha/2}G_k\|_2\|(1+|x|)^{-\alpha/2}g_k\|_2.
\]
Since $\|g_k\|^2_{L^2(\RN, \, (1+|x|)^{-\alpha})}\ge \int_{-\infty}^\infty |h_{\tilde k}(x_1)|^2(1+|x_1|)^{-\alpha}dx_1
\big(\int_0^1 |h_{0}(t)|^2dt\big)^{n-1}$, by  the estimate \eqref{e7.00} it follows that
\[
\|g_k\|^2_{L^2(\RN, \, (1+|x|)^{-\alpha})}
\geq C \max\{k^{-\alpha/2},k^{-1/2}\}.
\]
 Combining this with the above inequality yields
\[\|\chi_{[k,k+1)}(H)G_k\|_2
\geq C  \max\{k^{-\alpha/4},k^{-1/4}\}  \|G_k\|_{L^2(\RN, \, (1+|x|)^{\alpha}) }.\]
We also have  $ k^\lambda \|G_k\|_{L^2(\RN, \, (1+|x|)^{\alpha}) }
\geq  C\|\chi_{[k,k+1)}(H) G_k\|_{L^2(\RN, \, (1+|x|)^{\alpha}) } $ using the estimate  \eqref{e7.4}.
Thus  we have the estimate \eqref{e7.44}.

We apply the estimate  \eqref{e7.13}  to the function $G_k$ and combine the consequent estimate with   \eqref{e7.44}  to get
\begin{eqnarray}
\label{e7.14-1}
 k^{\alpha/4} \|\chi_{[k,k+1)}(H) G_k\|_2
&\leq& Ck^\lambda \min\big\{k^{\alpha/4},k^{1/4}\big\} \, \|\chi_{[k,k+1)}(H) G_k\|_2
\end{eqnarray}
with $C$ independent of $k$. Since
$
\langle G_k, \Phi_{\mu_0}\rangle
=\int_{\RN} |h_{k}(x_1)h_0(x_2)\cdots h_0(x_n)|^2(1+|x|)^{-\alpha}dx\neq 0
$ for $\mu_0=(k,0,\ldots,0)$, it follows that $\|\chi_{[k,k+1)}(H) G_k\|_2\neq 0$. Thus, \eqref{e7.14-1} implies
$k^{\alpha/4}
\le  Ck^\lambda \min\big\{k^{\alpha/4},k^{1/4}\big\}$
with $C$ independent of $k$.
Letting  $k\to \infty$ gives $
\lambda\geq \max\big\{ {(\alpha-1)}/{4},0\big\}
$
as desired.\end{proof}

\medskip

\begin{rem} Proposition~\ref{th6.1} can be used to give another proof of Proposition \ref{nec-2} provided that $1<\alpha<n$.
Indeed,  to the contrary,  suppose  that   \eqref{e7.1}  holds  with some $1<\alpha<n$ and   $\lambda <  (\alpha-1)/4$.
Now, for given $1<\alpha<n$ and   $\lambda <  (\alpha-1)/4$,
we can choose  a $p$ such that  $p>2n/(n-1)$, $\alpha> n(1-2/p)$, and $(\alpha-1)/4 > (n(1-2/p) -1)/4 >\lambda$.
By our choice of $p$ and H\"older's inequality  we have that  $f\in L^2({\mathbb R^n}, \, (1+|x|)^{-\alpha})$  if $f\in L^p$.
From this,  \eqref{e7.1}  implies that \eqref{ne7.1} holds  for all $f\in L^p$ and  $w(x)=(1+|x|)^{-\alpha}$.
Applying Proposition~\ref{th6.1}, we get   $
\lambda \ge n(1/2-1/p)/2-1/4,
$
  which is a contradiction.
   \end{rem}

 \noindent
{\bf Acknowledgments.}    P. Chen was supported by NNSF of China 11501583.
 X.T. Duong was supported by  the Australian Research Council (ARC) through the research
grant DP190100970.
D. He was supported by   NNSF of China (No. 11701583).
S. Lee  was supported by NRF (Republic of Korea) grant No. NRF2018R1A2B2006298.
 L. Yan was supported by the NNSF of China, Grant
No. 11521101 and 11871480, and by the Australian Research Council (ARC) through the research
grant DP190100970.
 P. Chen and L. Yan  would like to  thank Xianghong Chen, Ji Li and Adam Sikora for helpful  discussions.

\end{document}